\documentclass[journal,letterpaper,pdftex]{IEEEtran}
\usepackage{bm}
\usepackage{textcomp}
\usepackage{enumitem}
\usepackage{amsthm}
\usepackage{graphicx}
\usepackage{color}
\usepackage{subcaption}
\usepackage{float}
\usepackage[nocomma]{optidef}
\usepackage{dsfont}
\usepackage[font=small]{caption}
\usepackage{times}
\usepackage{wrapfig}
\usepackage{fancyhdr,graphicx,amsmath,amssymb}
\usepackage[ruled,vlined]{algorithm2e}
\usepackage{multirow}
\include{pythonlisting}
\usepackage{algorithmic}
\newcounter{algsubstate}


\usepackage{array}
\newcommand{\PreserveBackslash}[1]{\let\temp=\\#1\let\\=\temp}
\newcolumntype{C}[1]{>{\PreserveBackslash\centering}p{#1}}
\newcolumntype{R}[1]{>{\PreserveBackslash\raggedleft}p{#1}}
\newcolumntype{L}[1]{>{\PreserveBackslash\raggedright}p{#1}}
\allowdisplaybreaks

\newcommand{\bt}[1]{{\color{black}#1}}
\usepackage{cite}

\usepackage{xcolor}
\usepackage{amsmath}
\usepackage{amssymb}
\usepackage{mathtools}

\makeatletter
\def\BState{\State\hskip-\ALG@thistlm}
\makeatother
\usepackage{fixltx2e}

\usepackage{stfloats}

\begin{document}

%
\title{Competition in  Electric  Autonomous Mobility on Demand Systems}

\author{\IEEEauthorblockN{Berkay Turan \vspace{-0cm}\quad}
\and
\IEEEauthorblockN{Mahnoosh Alizadeh}
\vspace{-1cm}
}


%


\maketitle
\begin{abstract}This paper investigates the impacts of competition in autonomous mobility-on-demand systems. By adopting a network-flow based formulation, we first determine the optimal strategies of profit-maximizing platform operators in monopoly and duopoly markets, including the optimal prices of rides. Furthermore, we characterize the platform operator's profits and the consumer surplus. We show that for the duopoly, the \bt{equilibrium} prices for rides have to be symmetric between the firms. Then, in order to study the benefits of introducing competition in the market, we derive universal theoretical bounds on the ratio of prices for rides, aggregate demand served, profits of the firms, and consumer surplus between the monopolistic and the duopolistic setting. We discuss how consumers' firm loyalty affects each of the aforementioned metrics. Finally,  using the Manhattan network and demand data, we quantify the efficacy of  static pricing and routing policies and compare it to real-time model predictive policies.
\end{abstract}

%
\IEEEpeerreviewmaketitle
\newtheorem{proposition}{Proposition}
\newtheorem{assumption}{Assumption}
\newtheorem{remark}{Remark}
\newtheorem{corollary}{Corollary}[proposition]
\newtheorem{theorem}{Theorem}
\newtheorem{lemma}{Lemma}
\newcommand{\eqdef}{\vcentcolon=}
\newcommand{\lija}{\ell_{ij}^1}
\newcommand{\lijb}{\ell_{ij}^2}
\newcommand{\lmax}{\ell_{\max}}
\newcommand{\lijm}{\ell_{ij}^m}
\newcommand{\lijd}{\ell_{ij}^d}
\newcommand{\la}{\ell_1}
\newcommand{\lb}{\ell_2}
\newcommand{\lij}{\ell_{ij}}
\newcommand{\dijm}{\lambda_{ij}^{m}}
\newcommand{\dijd}{\lambda_{ij}^d}
\newcommand{\dij}{\lambda_{ij}}
\newcommand{\dbound}{\frac{(3\sigma-1)(3-\sigma)}{4(5-3\sigma)}\ell_{\max}}
\def\bt#1{{\color{black}#1}}
\def\btt#1{{\color{black}#1}}
\makeatletter
\def\blfootnote{\xdef\@thefnmark{}\@footnotetext}
\makeatother

\setlength{\abovedisplayskip}{3pt}
\setlength{\belowdisplayskip}{3pt}

\captionsetup{belowskip=-10pt}


\blfootnote{
This work is supported by the NSF Grants 1837125, 1847096. B. Turan and M. Alizadeh are with the Department of Electrical and Computer Engineering, University of California, Santa Barbara, CA, 93106 USA e-mail:\{bturan,alizadeh\}@ucsb.edu.}
\section{Introduction}

In the past decade, growing popularity of mobility-on-demand (MoD) platforms has extensively altered the paradigm of urban mobility. Owing to the rapid evolution of enabling technologies for autonomous driving and advancements in eco-friendly electric vehicles (EVs), it is possible for MoD platforms to employ self-driving electric vehicles and therefore preserve the benefits of private automobiles while reducing the consumption of non-renewable energy resources. Given this, the vision of an electric and autonomous mobility-on-demand (AMoD) fleet serving urban customers' mobility needs is gaining traction within the transportation industry, with multiple companies now heavily investing in AMoD technology \cite{companies}.

Unlike modern ride-sharing platforms that establish a two-sided market between drivers and customers such as Uber and Lyft and  rely only on pricing schemes to manage the demand-supply balance, self-driving technology allows AMoD systems to operate in a single-sided market in which the customers are directly served by the platform operator rather than the drivers. This allows the platforms to centrally manage their fleets of vehicles and hence efficiently dispatch them to where they are needed the most without the need to incentivize drivers, including \textit{rebalancing} the idle passenger-free vehicles throughout the network to match supply and demand.  Furthermore, autonomous vehicles can better exploit diversity in electricity prices for cheaper operation. 

In this paper, we study the effects of competition in electric AMoD systems that are operated by profit-maximizing platform operators.  Owing to the opportunities that autonomous electric vehicles create for efficient control schemes and cost-effective operation, it is possible for a single platform operator to provide cheap rides through optimizing the prices of rides for geographical load balancing as well as optimally routing and charging the fleet of electric vehicles.  However, a monopolistic market with a single AMoD provider is in general disadvantageous for customer welfare. Therefore, introduction of another AMoD service provider to the market results in firms competing over the customers, hence forcing them to charge fairer prices and provide a higher quality of service. Our primary goal is to investigate the optimal behaviour of the firms in a monopoly and duopoly and quantify the impacts of competition on the customers as well as the firms.


Our contributions can be summarized as follows:

\begin{itemize}
    \item  We  formalize the platform operator's profit maximization problem by adopting a static network-flow based model that captures the   characteristics of an AMoD fleet, and derive expressions for the ride prices, profits, and consumer surplus under the optimal static policy.
    \item We prove that if the competitors have identical costs, then the duopoly equilibrium prices have to be symmetric. We show that under a mild sufficient condition on maximum travel costs that can be met with electric vehicles, the duopoly prices in equilibrium are never larger than the optimal monopoly prices. Furthermore, we derive theoretical bounds for the ratio of prices, induced demand, profits, and consumer surplus  in   the  monopoly  and the duopoly equilibrium.
    \item 
    We study a real-time pricing and fleet management policy using model predictive control, and demonstrate the performance numerically on real network and demand data. 
\end{itemize}

\noindent
\textbf{Related work:} Research on AMoD systems concentrates on optimal fleet control, with a particular focus on rebalancing. Scholars have tackled the rebalancing problem  using queueing theoretical \cite{queuetheoretical}, fluidic \cite{fluidic}, network flow \cite{networkflow}, and Markovian models \cite{markovian}. These works develop time-invariant policies by relying on the steady-state of the system, which may not be able to address  real-world challenges such as  variability in customer arrivals or integrality requirements in real-time dispatch. As a consequence, studies aiming to develop efficient real-time control policies using model predictive control (MPC) have emerged in the literature, e.g., \cite{zhang_rossi_pavone,mpcmiao,mpcstochastic}. The authors of \cite{mpcmiao} design a data-driven MPC algorithm by predicting the future demand, whereas the authors of \cite{mpcstochastic} develop a stochastic MPC algorithm that leverages uncertain travel demand forecasts. In \cite{zhang_rossi_pavone}, the authors also consider a fleet of EVs and hence propose an MPC approach that optimizes vehicle routing and scheduling subject to energy constraints. EV charging problem in electric AMoD systems has also been studied using dynamic programming \cite{turan2019smart} and online heuristics \cite{tucker2019online}.
None of these studies however adopt a price-responsive demand model and exploit pricing to further optimize the system. Lately, benefits of joint pricing and rebalancing in AMoD systems have been demonstrated using macroscopic steady-state models \cite{wollenstein2020joint,turan2019dynamic} as well as microscopic dynamic models \cite{turan2019dynamic}.

Research on competition in ride-sharing markets is also relevant to ours. In terms of a broader scope on platform competition in two-sided markets, \cite{weyl2010imperfect} and \cite{rochet2003platform} introduce general frameworks and provide in-depth analysis. The impacts of single/multi-homing users on the market equilibria have been investigated in \cite{armstrong2007two}. Theoretical studies on dynamic platform competition \cite{dou2016dynamicplatformcomp} and spatial platform competition \cite{kodera2010spatialplatformcomp} in two sided markets further provide insights towards competition in ride-sharing markets. Besides these, scholars examine the competition between ride-sharing and taxis \cite{cramer2016ubertaxi,mcgregor2015ubertaxi}, where Uber is considered to be a monopoly. These works however do not capture the competition among ride-sharing platforms, yet ride-sharing markets are rather oligopolies in many countries\cite{ubernotmonopoly}. Accordingly, a recent work \cite{jiang2018ridesharing} presents a head-to-head comparison of
Uber, Lyft, and taxis using statistical methods.
Another line of work related to ours focuses on the benefits of spatial price discrimination \cite{bimpikis2019spatial} and dynamic pricing in ride-sharing networks \cite{castillo2017surge,banerjee2015pricing}. These however do not study a competitive market. Closest to our work is \cite{nikzad2017thickness}, which studies the effects of thickness (i.e., the mass of drivers) and competition on the equilibria of ride-sharing markets. It shows that competition always increases the welfare of the drivers, whereas it decreases the welfare of the customers
if the market is not sufficiently thick.

To the best of our knowledge, there is no existing work on competition in electric AMoD systems. Our study aims to form the bridge between AMoD and competition literature with our theoretical findings. \bt{We hope that the closed form bounds quantifying the impacts of competition would help investors make informed policy decisions about competing AMoD platforms and investing in efficient AMoD technologies.} 
\section{System Model and Problem Definition}\label{systemmodel}
\noindent
\textbf{Network and Demand Models:}
We consider two fleets of AMoD EVs operated by two competitors within a transportation network characterized by a \bt{complete} graph consisting of $\bt{\cal N} = \{1,\ldots,\bt{n}\}$ nodes. Each of these nodes can  serve as a trip origin or destination.  

\begin{figure}[t]
     \centering
     \hspace{-.5cm}
     \begin{subfigure}[b]{0.25\textwidth}
         \centering
         \includegraphics[width=1\textwidth]{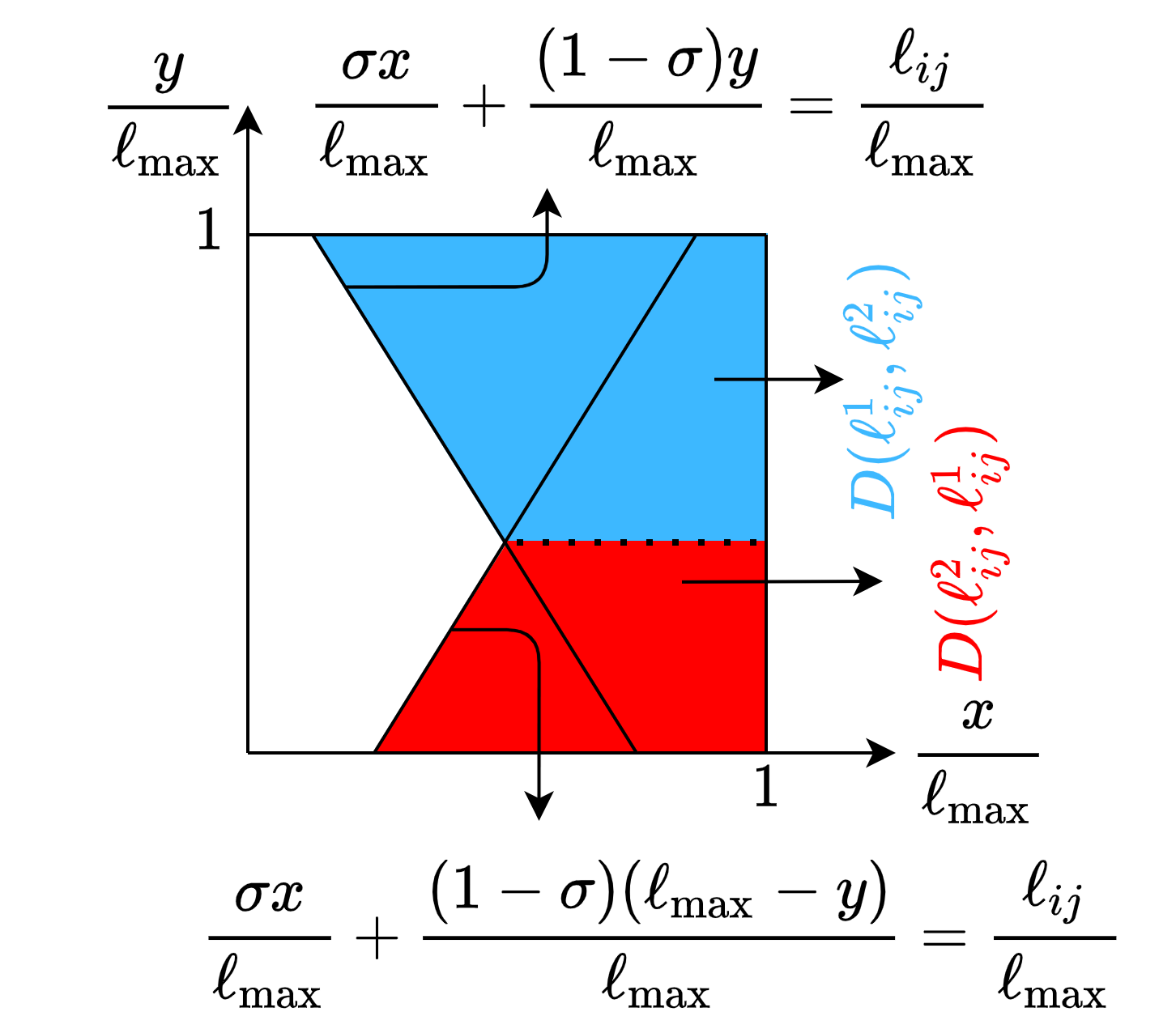}
         \caption{Duopoly demand functions}
         \label{fig:duopolydemand}
     \end{subfigure}
     \begin{subfigure}[b]{0.25\textwidth}
         \centering
         \includegraphics[width=1\textwidth]{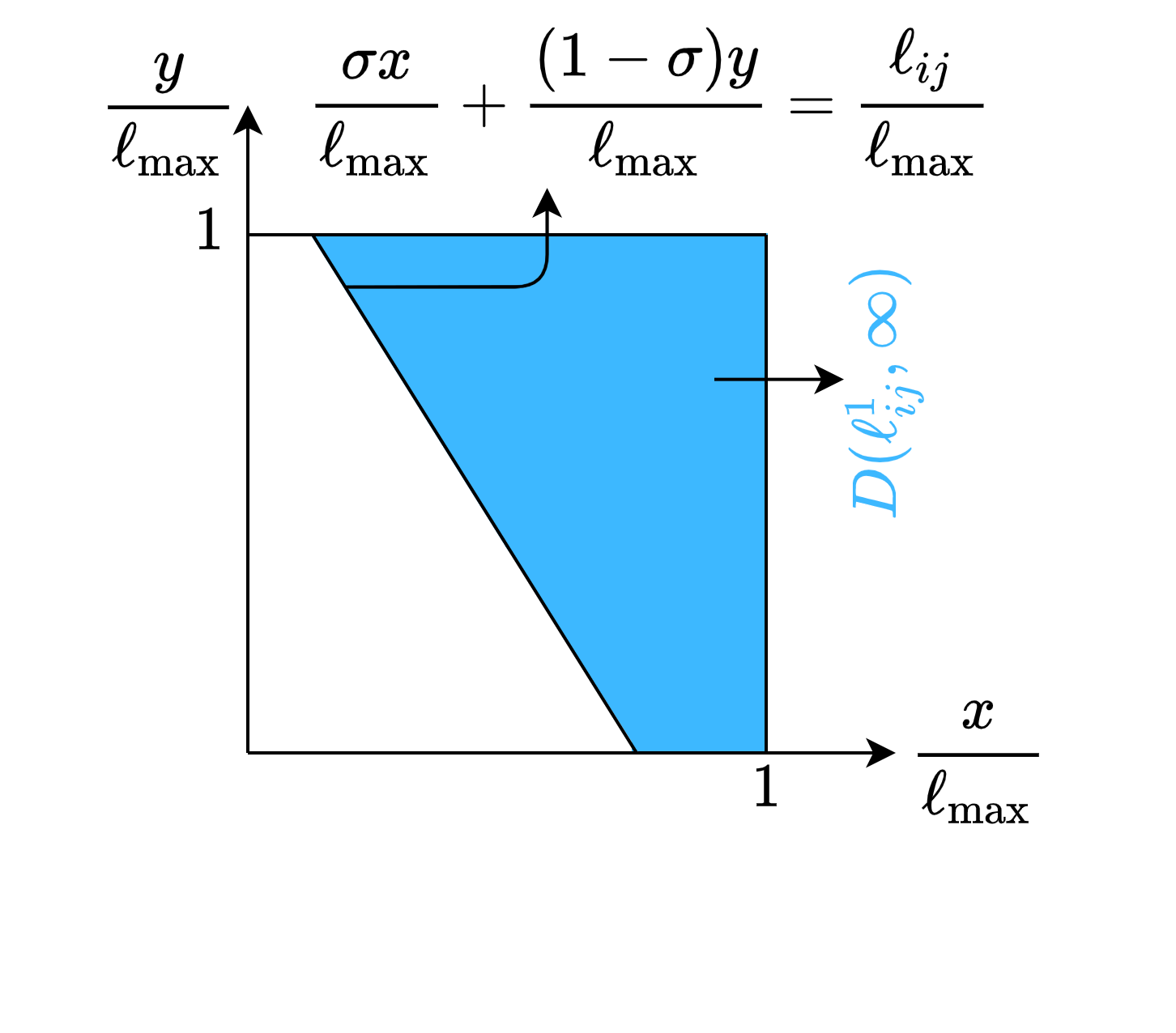}
         \caption{Monopoly demand function}
         \label{fig:monopolydemand}
     \end{subfigure}
     \hspace{-0.5cm}
     \vspace{7pt}
        \caption{Graphical illustration of the demand functions for (a) duopoly, and (b) monopoly. \btt{The axes correspond to the uniform random variables $x$ and $y$ scaled by $1/\lmax$.} In duopoly, the line $\sigma x+(1-\sigma)y=\ell_{ij}^1$ corresponds to the customers who earn $0$ pay-off buying a ride from firm 1, and the line $\sigma x+(1-\sigma)(\lmax-y)=\ell_{ij}^2$ corresponds to the customers who earn $0$ pay-off buying a ride from firm 2. As such, for the price tuple $(\ell_{ij}^1,\ell_{ij}^2)$, the blue shaded area corresponds to the \btt{demand function} for firm 1, whereas the red  corresponds to the \btt{demand function} for firm 2. Monopoly is the special case of duopoly, where the prices for rides set by firm 2 are set to infinity: $\ell_{ij}^2=\infty$.}
        \label{fig:demandfunctions}
\end{figure}
We study a discrete-time system with time periods normalized to integral units $t\in\{0,1,2,\dots\}$. In each period, potential riders of mass $\theta_{ij}$ seek rides between origin-destination (OD) pair $(i,j)$, where $\theta_{ii}=0$. We assume that customers have different valuations for riding with each firm, represented by the tuple $(v_1,v_2)$ where $v_f$ is the customer's valuation for firm $f$. To capture customer heterogeneity, we let $(v_1,v_2)\sim \cal V$, where $\cal V$ denotes the PDF of the joint distribution with support $[0,\ell_{\max}]^2$. Here, $\ell_{\max}$ is the maximum valuation of the customers for both firms, i.e. the maximum willingness to pay\footnote{For brevity of notation, we uniformly set $\ell_{\max}$ to be the maximum willingness to pay for all OD pairs without loss of generality. Our results can be derived in a similar fashion by replacing $\lmax$ with $\lmax^{ij}$, where $\lmax^{ij}$ is the maximum willingness to pay for OD pair $(i,j)$.}. To characterize the distribution $\cal V$, we adopt the  model proposed by \cite{nikzad2017thickness} and assume that the distribution of the
\pagebreak 
random variables $(v_1,v_2)$ is defined implicitly through:
\begin{align}
    \label{eq:valuations}
    &v_1=\sigma x+(1-\sigma)y,\\
   \label{eq:valuations2} &v_2=\sigma x+(1-\sigma)(\lmax-y),
\end{align}
where $x$ and $y$ are iid uniform random variables with support $[0,\ell_{\max}]$ and $\sigma\in [0,1]$. We refer to $x$ as the \emph{common value component} and $y$ as the \emph{idiosyncratic} component, with $\sigma$ as the measure of correlation over customers' preferences\bt{\footnote{\bt{In the monopolistic setting, $\sigma$ measures the correlation between customers' valuation of riding with the monopolistic firm and customers' valuation of riding with outside options (e.g., public transport).}}}. \bt{In particular, $x$ can be viewed as a customer's valuation of the ride itself and $y$ (or $\lmax-y$) can be viewed as a customer's valuation of firm 1 itself (or firm 2 itself). A customer is identified by the draws from distributions of $x$ and $y$, which are then mapped to that customer's valuations for riding with firms $1$ and $2$ via \eqref{eq:valuations} and \eqref{eq:valuations2}.}  A customer with valuations $(v_1,v_2)$ makes a decision upon observing the prices for rides. If the prices for rides between OD pair $(i,j)$ in period are set to be $\ell_{ij}^1$ and $\ell_{ij}^2$ by firm 1 and 2, respectively, the customer buys a ride from firm $f$ if $v_f-\ell_{ij}^f>0$ and $v_f-\ell_{ij}^f>v_{-f}-\ell_{ij}^{-f}$ (given firm $f$, $-f$ denotes the other firm), i.e., the customer gains a positive pay-off for purchasing a ride from firm $f$ and this pay-off is higher than the pay-off that the customer would gain by buying from the other firm. Otherwise they do not buy a ride from either of the firms and leave the system. Hence, for a price tuple $(\ell_{ij}^1,\ell_{ij}^2)$ for OD pair $(i,j)$, the induced mass of arrivals for firm $f$ is given by $\Theta_{ij}^f\eqdef\theta_{ij}D(\ell_{ij}^f,\ell_{ij}^{-f})$, where $D:[0,\ell_{\max}]^2\rightarrow[0,1]$ is the \emph{demand function} of customers which determines the fraction of customers that would buy a ride from firm $f$ upon observing the prices. This function has a simple geometric interpretation depicted in Figure~\ref{fig:demandfunctions}. \btt{We plot the demand function and the willingness to pay distribution as a function of ride prices for several values of $\sigma$ in the monopolistic setting in Figure~\ref{fig:demand_wtp}. Note that the demand function is concave if $\lija<(1-\sigma)\lmax$, is linear if $(1-\sigma)\lmax\leq\lija<\sigma\lmax$, and is convex if $\sigma\lmax\leq\lija$.}

\begin{figure}
    \centering
    \hspace*{-.5cm}\includegraphics[width=.55\textwidth]{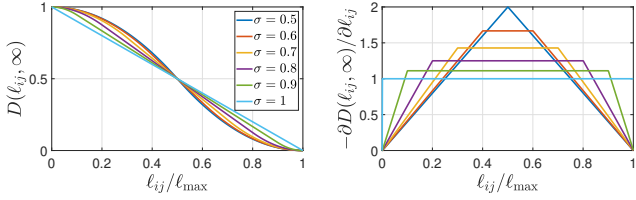}
    \caption{Demand function (left) and willingness to pay distribution (right) as a function of ride prices for several values of $\sigma\in[0.5,1]$.}
    \label{fig:demand_wtp}
\end{figure}

\noindent \textbf{Vehicle Model:} In order to best serve its customers and maximize its profits, each operator needs to dispatch its fleet, including vehicle routing and charging. 
To implicitly capture the effect of trip demand and the associated charging and routing decisions on the fleet size and hence the operational costs incurred by each operator, we assume that each vehicle in charging or trip-making mode has a per period operational cost of $\beta_c$ and $\beta_t$, respectively. A trip-making vehicle can either be occupied by a customer, which we refer to as a customer carrying vehicle; or can be empty, which we refer to as a rebalancing vehicle. \bt{We note that in this work, we set the capacity of a vehicle to be one passenger.} Furthermore, as the vehicles are electric, they have to sustain charge in order to operate, which needs to be purchased from the power grid. Without loss of generality, we assume there is a charging station placed at each node $i \in \bt{\cal N}$. To charge at node $i$, the operator pays a price of electricity $p_i$ per unit of energy.  
We assume that all EVs in the fleet have a battery capacity denoted as $e_{\max}\in \mathbb Z^+$; therefore, each EV has a discrete battery energy level $e \in \mathcal E$, where $\mathcal E = \{e\in \mathbb{N}| 0\leq e \leq e_{\max}\}$. In our discrete-time model, we assume each vehicle takes one period to charge one unit of energy and $\tau_{ij}$ periods to travel between OD pair $(i,j)$, while consuming $e_{ij}$ units of energy.


\noindent
\textbf{Platform Operator's Problem:} We consider a profit-maximizing AMoD operator that manages a  fleet of EVs that make trips to provide transportation services to customers. The operator's goal is to maximize profits by 1) setting prices for rides and hence managing customer demand at each node; 2) optimally operating the AMoD fleet (i.e., charging and  routing) to minimize operational and charging costs. Next, we study the static planning problem for both the monopoly and the duopoly settings in order to characterize the optimal static prices and to examine the effects of competition in electric AMoD systems.  


\section{Analysis of the Static Problem}\label{sec:static}
In this section, we establish and discuss the static planning problems considering a single operator (i.e., monopoly) and two competing operators (i.e., duopoly) in order to study the effect of competition in an electric AMoD system. We consider the fluid scaling of the network and characterize the static planning problem via a network flow formulation. 
The static problem is convenient for determining the optimal static pricing, routing, and charging policy of the platform operator. 

\vspace{-.1cm}
\subsection{Monopoly Static Planning Problem}
We define the monopoly to be the setting where the firm $2$ is removed. In order to make the comparison between the monopoly and the duopoly consistent, we keep the customer behaviour and the demand function $D$ same. Hence, removing firm $2$ from the system is equivalent to setting prices for rides posted by firm $2$ to be $\infty$, and the induced demand for rides for OD pair $(i,j)$ to be $D(\lija,\infty)$ for a given $\lija$.

The goal of the platform operator is to maximize its profits by setting prices for rides and making routing and charging decisions such that the induced demand is \bt{served.} Let $x_{ij}^e$ be the number of vehicles at node $i$ with energy level $e$ being routed to node $j$ and $x_{ic}^e$ be the number of vehicles charging at node $i$ and currently at energy level $e$. We state the platform operator's problem as follows:

\begin{subequations}
\label{eq:flowoptimization}
\begin{align}
\nonumber&\underset{x_{ic}^e,x_{ij}^e,\ell_{ij}^1}{\text{max}}
\nonumber& &\hspace{-.4cm}\sum_{i=1}^\bt{n}\sum_{j=1}^\bt{n}\theta_{ij}\ell_{ij}^1D(\ell_{ij}^1,\infty)\\
& & &\hspace{-.6cm}-\sum_{i=1}^\bt{n}\hspace{-.1cm}\sum_{e=0}^{e_{\max}-1}\hspace{-.1cm} (\beta_c{+}p_i) x_{ic}^e
\label{eq:staticobjective}{-}\beta_t \sum_{i=1}^\bt{n}\sum_{j=1}^\bt{n}\sum_{e=e_{ij}}^{e_{\max}} x_{ij}^e\tau_{ij} \\
& \text{subject to}
\label{eq:staticconst1}& &\theta_{ij}D(\ell_{ij}^1,\infty) \leq \sum_{e=e_{ij}}^{e_{\max}}x_{ij}^e \quad\forall i,j \in \bt{\cal N},\\
\nonumber& & & x_{ic}^e+\sum_{j=1}^\bt{n} x_{ij}^e=x_{ic}^{e-1}+\sum_{j=1}^\bt{n} x_{ji}^{e+e_{ji}},\\
\label{eq:staticconst2}& & &\hspace{3.5cm}\forall i\in\bt{\cal N},\;\forall e\in\mathcal E,\\
\label{eq:staticconst5}& & & x_{ic}^{e_{\max}}=0,\quad \forall i\in \bt{\cal N},\\
\label{eq:staticconst6}& & & x_{ij}^e=0,\quad \forall e<e_{ij},\; \forall  i,j\in \bt{\cal N},\\
\label{eq:staticconst7}& & & x_{ic}^e\geq 0, \; x_{ij}^e\geq 0,\; ~\forall i,j\in \bt{\cal N},\;\forall e\in\mathcal E,\\
\label{eq:staticconst8}& & & x_{ic}^e=x_{ij}^e=0,\quad \forall e\notin\mathcal E,\; \forall  i,j\in \bt{\cal N}.
\end{align}
\end{subequations}

\bt{The objective function \eqref{eq:staticobjective} corresponds to the profits earned by the firm per period. In particular, the first term in \eqref{eq:staticobjective}} accounts for the aggregate revenue the platform generates by providing rides for $\theta_{ij}D(\ell_{ij}^1,\infty)$ number of riders with a price of $\lija$.
The second term is the operational and charging costs incurred by the charging vehicles, and the last term is the operational costs of the trip-making vehicles.

The constraint \eqref{eq:staticconst1} requires the platform to operate at least as many vehicles to serve all the induced demand between any two nodes $i$ and $j$
(The rest are the vehicles travelling without passengers, i.e., rebalancing vehicles). We will refer to this as the {\it demand satisfaction constraint}. \bt{We let $\lambda_{ij}$ be the dual variable associated with \eqref{eq:staticconst1} and $\lambda_{ij}^m$ be the optimal dual variable.}
The constraint \eqref{eq:staticconst2} is the  \textit{flow balance constraint} for each node and each battery energy level, which restricts the number of available vehicles at node $i$ and energy level $e$ to be the sum of arrivals from all nodes and vehicles that are charging with energy level $e-1$. The constraint \eqref{eq:staticconst5} ensures that the vehicles with full battery do not charge further, and the constraint \eqref{eq:staticconst6} ensures the vehicles sustain enough charge to travel between OD pair $(i,j)$.

\bt{It is worthwhile to mention that unlike traditional minimum-cost flow problems, where the objective is to minimize total travel cost, the objective of \eqref{eq:flowoptimization} is to maximize the total revenue minus the costs, i.e., profits. Furthermore, in traditional minimum-cost flow problems, demand elasticity in response to price is not explicit and the elasticity is often modeled in response to travel times \cite{dafermos1982general,yang1997sensitivity}, whereas the explicit dependency of the induced demand to prices via $D(\lija,\infty)$ results in a more challenging task. The prices affect the induced demand, which affects the routing decisions and this causes a complex interplay between the decision variables.}

\noindent\textbf{Optimal Pricing:} The prices for rides  are a crucial component of the profits generated. The next proposition highlights how the optimal prices $\ell_{ij}^{m}\eqdef \ell_{ij}^{1*}$ for rides are related to the network parameters, prices of electricity, and the  operational costs.  In the following results, we investigate this interconnection by providing upper bounds on the prices that a profit-maximizing monopolist may charge customers, as well the corresponding profits generated. We highlight the fact that the monopolist's profits are in fact a decreasing function of the optimal prices for rides. The higher the monopolist has to charge its customers, the lower its generated profits. This could be a motivation for the monopolist to invest in efficient vehicle technology and cheap charging solutions.

\begin{proposition}\label{prop:monomarginalprices}
Define 
$$\overline{\lambda}_{ij}\eqdef\beta_t(\tau_{ij}+\tau_{ji})+e_{ij}(p_j+\beta_c)+e_{ji}(p_i+\beta_c).$$ 
Let $\lambda_{ij}^{m}$ be \bt{the optimal} dual variable corresponding to the demand satisfaction constraint \bt{\eqref{eq:staticconst1}} for OD pair $(i,j)$. The optimal monopoly prices $\ell_{ij}^m$ are:
\begin{equation}
\label{eq:optimalprices}
    \lijm{=}\left\{
\begin{array}{ll}
      \frac{\lambda_{ij}^{m}{+}\sqrt{({\lambda_{ij}^m})^2+6\sigma(1{-}\sigma)\lmax^2}}{3}&,\frac{\dijm}{\lmax}<\frac{3-5\sigma}{2} \\[1ex]
      \frac{(1+\sigma)\lmax+2\dijm}{4}&,\frac{3-5\sigma}{2}\leq \frac{\dijm}{\lmax}< \frac{3\sigma-1}{2} \\[1ex]
      \frac{2\dijm+\lmax}{3}&,\frac{3\sigma-1}{2}\leq \frac{\dijm}{\lmax}\leq1. \\
\end{array} 
\right. 
\vspace{-.2cm}
\end{equation}   
These prices can be upper bounded by:
\begin{equation}
\label{eq:monoboundprices}
    \lijm\leq\left\{
\begin{array}{ll}
      \frac{\overline{\lambda}_{ij}+\sqrt{(\overline{\lambda}_{ij})^2+6\sigma(1-\sigma)\lmax^2}}{3}&,\frac{\overline{\lambda}_{ij}}{\lmax}< \frac{3-5\sigma}{2} \\[1ex]
      \frac{(1+\sigma)\lmax+2\overline{\lambda}_{ij}}{4}&,\frac{3-5\sigma}{2}\leq \frac{\overline{\lambda}_{ij}}{\lmax}< \frac{3\sigma-1}{2} \\[1ex]
      \frac{2\overline{\lambda}_{ij}+\lmax}{3}&,\frac{3\sigma-1}{2}\leq \frac{\overline{\lambda}_{ij}}{\lmax}\leq1. \\
\end{array} 
\right. 
\vspace{-.2cm}
\end{equation} 
\end{proposition}
The proof can be found in Appendix~\ref{sec:proofprop1}. We can interpret the dual variables $\lambda_{ij}^{m}$ as the cost of providing a single ride between $i$ and $j$ to the platform.  In the worst case scenario, every single requested ride from node $i$ requires rebalancing and charging both at the origin and the destination. Hence the upper bounds on \eqref{eq:monoboundprices} include the operational costs of passenger-carrying, rebalancing and charging vehicles (both at the origin and the destination); and the energy costs of both passenger-carrying and rebalancing trips multiplied by the price of electricity at the trip destinations (This is exactly what $\overline{\lambda}_{ij}$ consists of).

Similar to the taxes applied on products, whose burden is shared among the supplier and the customer; the costs associated with rides are shared among the platform operator and the riders (which is why the price paid by the riders include some fraction of the cost of the ride).

\bt{We note that if the optimal dual variables $\lambda_{ij}^m$ fall in the region $[(3-5\sigma)\lmax/2,(3\sigma-1)\lmax/2]$, then the optimal prices given by \eqref{eq:optimalprices} fall in the region $[(1-\sigma)\lmax,\sigma\lmax]$. In this region, the demand function $D(\lijm,\infty)$ is linear. Hence, the optimization problem \eqref{eq:flowoptimization} (with the additional constraint $(1-\sigma)\lmax\leq\lija\leq\sigma\lmax,~\forall i,j\in{\cal N}$, without losing global optimality) becomes a convex quadratic program and can be solved in polynomial time. The following assumptions guarantee this:} 
\bt{
\begin{assumption}
\label{ass:sigma}
Assume that $\sigma\geq 3/5$, i.e., the customers' preferences over the two firms are highly correlated.
\end{assumption}
}
\begin{assumption}
\label{ass:sigmalambda}
We assume $\underset{i,j}{\max} ~\overline{\lambda}_{ij}\leq\frac{(3\sigma-1)(3-\sigma)}{4(5-3\sigma)}\ell_{\max}$ as an upper bound on the maximum cost of a ride in the network.
\end{assumption}
\begin{remark}
\bt{Assumption~\ref{ass:sigma} implies that at least $3/5$ ($60\%$) of the customers' valuations between the firms are correlated. Higher correlation implies that riders' valuations of the rides provided by a firm depend less on the identity of the firm. This is reasonable for ride-sharing platforms, where majority of the customers decide depending heavily on the price rather than the identity of the firm.\\
Assumption~\ref{ass:sigmalambda} imposes an upper bound on the maximum cost of a ride. }This can be satisfied in practice, especially with electric vehicles. Observe that the bound is increasing with $\sigma$, hence it is tightest when $\sigma=3/5$. To give numbers with a simple calculation, consider a network with farthest OD pair of $15$ miles and $30$ minutes away (with average speed $30 mph$), $\sigma=3/5$ and $\lmax=\$50$. An average EV consumes 34kWh energy to drive for 100 miles. For an average price of electricity of $\$0.11$ per kWh and a charger with $\bt{20kW}$ charging speed, the EV charges $10kWh$ in $30$ minutes for $\$1.1$, that allows for 30 miles of range. If we \bt{amortize} the cost of a very expensive EV of $\$100k$ over 5 years, we get per minute operational cost of $\$0.04$. In total, to do the trip and the rebalancing, the vehicle drives for $30$ miles for $1$ hour and charges for $30$ minutes. In total, this yields a cost of $90\times\$0.04+\$1.1=\$4.7=\underset{i,j}{\max} ~\overline{\lambda}_{ij}\leq\frac{(3\sigma-1)(3-\sigma)}{4(5-3\sigma)}=\$7.5$. Whereas the fuel for gasoline vehicles costs about 4 times more (around $\$0.16$ per mile), which would yield $\underset{i,j}{\max} ~\overline{\lambda}_{ij}=\$8.00$.
\end{remark}

Next, we relate the optimal prices $\ell_{ij}^{m}$ to the profits generated by the operator and the consumer surplus. \bt{The profits are defined by the objective function in \eqref{eq:staticobjective}. The consumer surplus is defined as the difference between the price that customers pay and the price that they are willing to pay, i.e., the aggregate pay-off of the customers.}
\begin{proposition}\label{prop:monoprofitcs}
\bt{Suppose that Assumptions~\ref{ass:sigma} and \ref{ass:sigmalambda} hold.} With the optimal monopoly prices $\lijm$, the profits per period are:
\begin{equation}
    \label{eq:monoprofits}
    P^m=\sum_{i=1}^\bt{n}\sum_{j=1}^\bt{n}\frac{\theta_{ij}}{4\sigma\ell_{\max}}(\ell_{\max}(1+\sigma)-2\ell_{ij}^m)^2.
\end{equation}
The consumer surplus with the optimal prices is:
\begin{equation}
    \label{eq:monocs}
    \textnormal{CS}^m{=}\sum_{i=1}^\bt{n}\sum_{j=1}^\bt{n}\theta_{ij}\frac{\lmax(\sigma^2{+}\sigma{+}1){-}3\lijm(1{+}\sigma{-}\frac{\lijm}{\lmax})}{6\sigma}.
\end{equation}
\end{proposition}
The proof can be found in Appendix~\ref{app:profcsmono}. Notice that the profits in \eqref{eq:monoprofits} are decreasing as the prices for rides increase. Thus expensive rides generate less profits compared to the cheaper rides and it is more beneficial if the optimal dual variables $\lambda_{ij}^{m}$ are small and prices are close to $\ell_{\max}(1+\sigma)/4$. Thus, the operator has incentive to use more efficient routing and charging policies so they can lower ride prices as much as possible. Moreover, by computing $\frac{\partial \textnormal{CS}^m}{\partial \ell_{ij}^m}$ using \eqref{eq:monocs}, one identifies that lower prices generate higher consumer surplus, which is an intuitive result.


\vspace{-.1cm}
\subsection{Duopoly Static Planning Problem}
We study the duopoly as a game between two firms. At a high level, the game is described by firm $f$ observing firm $-f$'s prices and solving the optimization problem \eqref{eq:flowoptimization} (by considering firm $-f$'s prices to be $\ell_{ij}^{-f}$ rather than $\infty$ for the demand function).
We consider two competitors with identical operational costs $\beta_t$ and $\beta_c$, and study the optimal pricing strategy when the firms are at an equilibrium. In an equilibrium, no firm benefits from unilaterally changing the prices for any number of OD pairs (and as a result the optimal solution to their static planning problem). Given $\{\ell_{ij}^{-f}\}_{ \forall i,j\in{\bt{\cal N}}}$, the best response of firm $f$ is the best pricing, routing and charging strategy of $f$, which is the solution of \eqref{eq:flowoptimization} (with $\ell_{ij}^{-f}$ instead of $\infty$ in the demand function). Since the operational costs and the prices of electricity are identical for both of the firms, their best response to their competitor's prices are the same. As such, it is intuitive that there exists an equilibrium in which both firms set the prices equal ($\ell_{ij}^f=\ell_{ij}^{-f},\forall i,j\in {\bt{\cal N}}$), and we  show that this is in fact the case. Such an equilibrium is commonly referred to as a \emph{symmetric} duopoly equilibrium. Furthermore, we  show that  no asymmetric equilibria can exist under this setting, i.e., identical firms will not set different prices for the same OD pair at equilibrium.

Let the following static planning problem characterize the state in which both firms serve equal number of customers for all OD pairs and have identical pricing strategies:
\begin{subequations}
\label{eq:duopolyflowoptimization}
\begin{align}
\nonumber&\underset{x_{ic}^e,x_{ij}^e,\ell_{ij}^1}{\text{max}}
\nonumber& &\hspace{-.5cm}\sum_{i=1}^\bt{n}\sum_{j=1}^\bt{n}\theta_{ij}\ell_{ij}^1D(\ell_{ij}^1,\lijb)\Big\rvert_{\lija=\lijb}\\
& & &\hspace{-.7cm}-\sum_{i=1}^\bt{n}\hspace{-.1cm}\sum_{e=0}^{e_{\max}-1} \hspace{-.1cm}(\beta_c+p_i) x_{ic}^e
\label{eq:staticobjectiveduo}{-}\beta_t \sum_{i=1}^\bt{n}\sum_{j=1}^\bt{n}\sum_{e=e_{ij}}^{e_{\max}} \hspace{-.1cm}x_{ij}^e\tau_{ij} \\
& \text{subject to}
\label{eq:staticconst1duo}&  &\hspace{-.2cm}\theta_{ij}D(\ell_{ij}^1,\lijb)\Big\rvert_{\lija=\lijb} \hspace{-.1cm}\leq \hspace{-.1cm}\sum_{e=e_{ij}}^{e_{\max}}\hspace{-.1cm}x_{ij}^e, \quad\forall i,j \in \bt{\cal N},\\
\nonumber& & &\eqref{eq:staticconst2}-\eqref{eq:staticconst8}.
\end{align}
\end{subequations}
\bt{We note that the optimization problem \eqref{eq:duopolyflowoptimization} is in general non-convex due to $D(\lija,\lijb)$.} Since there are no constraints on the fleet size and furthermore prices that control the demand are decision variables, a feasible solution to the above optimization problem always exists.
Moreover, the optimal solution to \eqref{eq:duopolyflowoptimization} specifies an equilibrium of the duopoly.

\begin{proposition}\label{prop:unique}
    \bt{Suppose that $\sigma\geq 1/2$ and Assumption~\ref{ass:sigmalambda} holds.} The firms are in an equilibrium when their routing, charging, and symmetric pricing strategy follows the solution of \eqref{eq:duopolyflowoptimization}.
\end{proposition}
\bt{\textit{Proof outline: }We first determine the optimal pricing strategy $\{\ell_{ij}^d\}_{i,j\in{\cal N}}$ of \eqref{eq:duopolyflowoptimization} using the first and second order optimality conditions (similar to proof of Proposition 1). Then, \btt{by stating the first order optimality condition for firm $f$ we show that} when firm $-f$ sets prices as $\ell_{ij}^{-f}=\ell_{ij}^d$, $\forall i,j\in{\cal N}$, then the best response of firm $f$ is to set $\ell_{ij}^f=\ell_{ij}^d$, $\forall i,j\in{\cal N}$. Hence, they are in an equilibrium.\hfill$\square$}

The complete proof can be found in Appendix~\ref{app:unique}. Accordingly, there exists a duopoly equilibrium characterized as the optimal solution of \eqref{eq:duopolyflowoptimization}, in which the firms set identical prices. The optimal solution to \eqref{eq:duopolyflowoptimization} is however not necessarily unique and there can be many solutions yielding the same profits. For instance, 
if $p_i=p_j,\forall i,j\in\bt{\cal N}$, then the optimal charging strategy is not unique. We let $\{\ell_{ij}^d\}_{i,j\in\bt{\cal N}}$ to be the \bt{equilibrium} prices determined as an optimal solution of \eqref{eq:duopolyflowoptimization}
and say that the firms are in a symmetric duopoly equilibrium as long as $\ell_{ij}^1=\ell_{ij}^2=\ell_{ij}^d, \forall i,j\in\bt{\cal N}$. Furthermore, in the next proposition, we state that if both firms serve all OD pairs, equilibrium prices can not be asymmetric.

\begin{proposition}\label{prop:equilibrium}
\bt{Suppose that $\sigma\geq 1/2$ and Assumption~\ref{ass:sigmalambda} holds.} There exists no asymmetric equilibrium prices, in which both firms serve nonzero demand for all OD pairs with nonzero potential \bt{riders}.
\end{proposition}
\bt{\textit{Proof outline: }We let $\ell_{ij}^1=\ell_{ij}^2+\delta$ for some $\delta>0$ and show by contradiction that the first-order optimality condition can not simultaneously be satisfied for both firms. Since the demand function $D(\lija,\lijb)$ has different expressions for $\lija\leq(1-\sigma)\lmax$ and $\lija>(1-\sigma)\lmax$, we separately study three cases: (i) $\lija,\lijb\leq (1-\sigma)\lmax$, (ii) $\lija,\lijb> (1-\sigma)\lmax$, and (iii) $\lija>(1-\sigma)\lmax$, $\lijb\leq (1-\sigma)\lmax$. For all cases, we first assume that the first-order optimality condition hold for both firms and bound the difference between the dual variables leading to $\lija$ and $\lijb$ in terms of $\delta$. For cases (i) and (ii), we show by using the bound on the dual variables that if the first-order condition for firm 2 is satisfied (i.e., is equal to 0), then the first-order condition for firm 1 is always less than 0, which is a contradiction. For case (iii), we show that with first-order condition satisfying prices, $\lijb+\delta$ is always less than $\lija$, which is a contradiction. \hfill$\square$}

The complete proof is provided in Appendix~\ref{app:equilibrium}. As we have identified that the duopoly can only be in a symmetric equilibrium, we analyze the effects of competition in state of a symmetric equilibrium.

The next set of results characterize the effects of competition on the ride prices, the operators' profits, the total societal ride demand served, and the consumer surplus. In the first result, we provide lower and upper bounds on the price reduction the customers will see with the introduction of the second firm and moving from a monopoly to a symmetric duopoly equilibrium.
\begin{proposition}
\label{prop:duomarginalprices}
\bt{Suppose that Assumptions~\ref{ass:sigma} and \ref{ass:sigmalambda} hold.} Let $\overline{\lambda}_{ij}$ be defined as in Proposition~\ref{prop:monomarginalprices}. Define
$$\Delta_1(\lambda_{ij})\eqdef4\lmax^2+(2\lambda_{ij}+(15\sigma-3)\ell_{\max})(2\lambda_{ij}+(1-\sigma)\ell_{\max}),$$
$$\Delta_2(\lambda_{ij})\eqdef2(\sigma\lmax-\lambda_{ij})^2+2(\lmax-\lambda_{ij})^2+11(\sigma-1)^2\lmax^2.$$
Let $\dijd$ be the optimal dual variable corresponding to the demand satisfaction constraint \eqref{eq:staticconst1duo}. The symmetric duopoly \bt{equilibrium} prices are determined as:
\begin{equation}
    \label{eq:duooptprices}
    \hspace{-.05cm}\ell_{ij}^d=\left\{
\begin{array}{ll}
      \frac{(3-5\sigma)\lmax+2\lambda_{ij}^d+\sqrt{\Delta_1(\dijd)}}{8}& ,\frac{\lambda_{ij}^d}{\lmax}\leq\frac{3(1-\sigma)^2}{2(\sigma+1)} \\[1ex]
      \frac{(5-3\sigma)\lmax+2\lambda_{ij}^d-\sqrt{\Delta_2(\dijd)}}{4}&,\textnormal{o.w.}, \\
\end{array} 
\right.
\end{equation}
Moreover, denote the difference between optimal monopoly and symmetric duopoly \bt{equilibrium} prices for OD pair $(i,j)$ as $\Delta \ell_{ij}\eqdef\ell_{ij}^m-\ell_{ij}^d$. Then:
\begin{equation}
    \label{eq:gainlowerbound}
  \Delta \ell_{ij}{\geq}\left\{
\begin{array}{ll}
       \frac{(7\sigma-1)\ell_{\max}-2\overline{
    \lambda}_{ij}-\sqrt{\Delta_1(\overline{\lambda}_{ij})}}{8}& ,\frac{\overline{\lambda}_{ij}}{\lmax}\leq\frac{3(1-\sigma)^2}{2(\sigma+1)} \\[1ex]
      \frac{(4\sigma-4)\lmax-2\overline{\lambda}_{ij}+\sqrt{\Delta_2(\overline{\lambda}_{ij})}}{4}&,\textnormal{o.w.}, \\
\end{array} 
\right.
\vspace{-.2cm}
\end{equation}
and
\begin{equation}
    \label{eq:gainupperbound}
    \Delta\ell_{ij}{\leq}\frac{(7\sigma{-}1)\ell_{\max}{+}4\overline{\lambda}_{ij}{-}\ell_{\max}\sqrt{{-}15\sigma^2{+}18\sigma{+}1}}{8}.
\end{equation}
\end{proposition}
\bt{\textit{Proof outline: }\btt{We state} the first and the second order optimality conditions on \eqref{eq:duopolyflowoptimization} to get the duopoly equilibrium prices. To lower bound the price difference, we evaluate the monopoly prices at $\lambda_{ij}^m=0$ and the duopoly equilibrium prices at $\lambda_{ij}^d=\dbound$ (and to upper bound, vice versa).\hfill$\square$}

The complete proof can be found in Appendix~\ref{app:duomarginalprices}.
An interesting observation is how $\sigma$ affects the prices. For the optimal monopoly prices, $\partial \lijm/\partial \sigma>0$, i.e., the monopolist serving a population with higher $\sigma$ charges more for the rides with identical costs (i.e., identical $\dijm$). \bt{The reason is that  larger $\sigma$ shifts the distribution of customers' valuations for the monopolist from intermediate to extreme values (as $\sigma$ increases from $1/2$ to $1$, the distribution shifts from triangular to uniform). This shift in the distribution modifies the demand function $D(\lija,\infty)$, which leads to an increase on the optimal prices. Simply put, larger $\sigma$, i.e., lack of firm loyalty, leads to an increase in the prices for the monopoly.}
On the contrary for the duopoly \bt{equilibrium} prices, $\partial \lijd/\partial\sigma<0$. That is, the duopoly serving a population with higher $\sigma$ charges less for the rides with identical costs (i.e., identical $\dijd$). The intuition behind is that larger $\sigma$ indicates a lack of firm loyalty (when $\sigma=1$, the customers buy from the firm that offers lower prices).  Hence, higher $\sigma$ strengthens the competition and causes the firms to charge less. The reader can observe that when $\sigma=1$, $\lijd=\dijd$, i.e., the \bt{equilibrium} prices are equal to the costs of providing the rides to the platform, which is the lowest the firms can go without losing money but make no profit.

Observe that the lower bounds in \eqref{eq:gainlowerbound} are decreasing functions of $\overline{\lambda}_{ij}$. Given the maximum value of $\overline{\lambda}_{ij}$ equal to $\overline{\lambda}_{ij}=\frac{(3\sigma-1)(3-\sigma)}{4(5-3\sigma)}\ell_{\max}$, the lower bound on the price difference is $0$. Hence, we can conclude that the duopoly prices are never higher than the monopoly prices, for all OD pairs.

Proposition~\ref{prop:duomarginalprices} characterizes the effect of competition on the prices depending on the network parameters and therefore the dual variables. The next series of results aim to determine universal bounds on the ratio of prices, induced demand, profits and consumer surplus in the monopoly and the duopoly, {\it independent of the network parameters}.

\begin{proposition}[Price Bounds]\label{cor:pricebounds}
\bt{Suppose that Assumptions~\ref{ass:sigma} and \ref{ass:sigmalambda} hold.} For all OD pairs, the optimal monopoly prices obey the following:
\begin{equation}
    \label{eq:monopolypricesbound}
    {2\lmax}/{5}\leq\underline{\ell}^m\leq\ell_{ij}^m\leq\overline{\ell}^m\leq{3\lmax}/{4},
    \end{equation}
where $\underline{\ell}^m\eqdef\frac{1+\sigma}{4}\lmax$ and $\overline{\ell}^m\eqdef \frac{7+14\sigma-9\sigma^2}{40-24\sigma}\lmax$. Furthermore, the symmetric duopoly \bt{equilibrium} prices obey:
\begin{equation}
    \label{eq:duopolypricesbound}
    0\leq\underline{\ell}^d\leq\ell_{ij}^d\leq\overline{\ell}^d\leq{\lmax}/{2},
\end{equation}
where $\underline{\ell}^d\eqdef\frac{3-5\sigma+\sqrt{-15\sigma^2+18\sigma+1}}{8}\lmax$ and $\overline{\ell}^d\eqdef \frac{1+\sigma}{4}\lmax$.
Moreover for all OD pairs $(i,j)$, the ratio between the \bt{symmetric }duopoly {equilibrium prices} and the \bt{optimal} monopoly prices obey the following:
\begin{equation}
\label{eq:pricesratio}
    \frac{\underline{\ell}^d}{\overline{\ell}^m}\leq\frac{\ell_{ij}^d}{\ell_{ij}^m}\leq\frac{\overline{\ell}^d}{\underline{\ell}^m}=1.
\end{equation}
\end{proposition}
\bt{\textit{Proof outline: }The proof is done by evaluating the optimal monopoly prices given by \eqref{eq:optimalprices} at $\lambda_{ij}^m=0$ and $\lambda_{ij}^m=\dbound$ as well as the duopoly equilibrium prices given by \eqref{eq:duooptprices} at $\lambda_{ij}^d=0$ and $\lambda_{ij}^d=\dbound$ to get the bounds on the prices in terms of $\sigma$. Then, we impose the condition $\sigma\in[3/5,1]$ to get the uniform bounds.}

The complete proof can be found in Appendix~\ref{app:pricebounds}.
An observation is that increasing $\sigma$ increases both the upper and the lower bounds for the optimal monopoly prices, whereas decreases the lower bound on the duopoly \bt{equilibrium} prices and increases the upper bound. This is because for the optimal monopoly prices, $\partial \lijm/\partial \sigma>0$. However, because it strengthens the competition between the firms in the duopoly, it can cause the prices to go much lower, hence decreasing the lower bound (when $\sigma=1$: if $\lambda_{ij}^d=0$, then $\lijd=0$). The upper bound on the duopoly \bt{equilibrium} prices still increases, because according to Assumption~\ref{ass:sigmalambda} a larger $\sigma$ permits a larger $\lambda_{ij}^d$ and hence higher prices. Consequently, the upper bound on the price ratio is always 1 independent of $\sigma$ while the lower bound is decreasing with $\sigma$.

The next result characterizes the effect of competition on the total customer demand for rides that are served by either firm. We show that the aggregate demand served by the duopoly is at least equal to and can be up to $4$ times higher than the demand served by the monopoly.

\begin{proposition}[Demand Bounds]
\label{cor:demandfunctionbounds}
\bt{Suppose that Assumptions~\ref{ass:sigma} and \ref{ass:sigmalambda} hold.} For all OD pairs $(i,j)$, the monopoly demand functions evaluated at the optimal monopoly prices obey:
\begin{equation}
    \label{eq:monopolydemandbounds}
    {1}/{4}\leq\underline{D}^m\leq  D(\ell_{ij}^m,\infty)\leq\overline{D}^m\leq{2}/{3},
\end{equation}
where $\underline{D}^m\eqdef \frac{13-3\sigma^2-6\sigma}{40\sigma-24\sigma^2}$ and $\overline{D}^m\eqdef\frac{1+\sigma}{4\sigma}$. The duopoly demand functions at the duopoly \bt{equilibrium} prices obey:
\begin{equation}
    \label{eq:duopolydemandbounds}
    {1}/{4}\leq\underline{D}^d\leq D(\ell_{ij}^d,\ell_{ij}^d)\leq\overline{D}^d\leq{1}/{2},
\end{equation}
where $\underline{D}^d\eqdef\frac{1}{4\sigma}$ and $\overline{D}^d\eqdef \frac{1}{2}-\frac{\left(-(1+\sigma)+\sqrt{-15\sigma^2+18\sigma+1}\right)^2}{128\sigma(1-\sigma)}$.
Furthermore, the ratio between the total demand served between any OD pair $\frac{2D(\lijd,\lijd)}{D(\lijm,\infty)}$ obeys the following:
\begin{equation}
    \label{eq:monoduodemandratio}
    1\leq\frac{2}{1+\sigma}=\frac{2\underline{D}^d}{\overline{D}^m}\leq\frac{2D(\lijd,\lijd)}{D(\lijm,\infty)}\leq\frac{2\overline{D}^d}{\underline{D}^m}\leq 4
\end{equation}
\end{proposition}
\bt{\textit{Proof outline: }The proof is done by evaluating the demand functions for the monopoly and the duopoly at the price bounds given by \eqref{eq:monopolypricesbound} and \eqref{eq:duopolypricesbound}, and then imposing the condition $\sigma\in[3/5,1]$ to get uniform bounds.}

The complete proof can be found in Appendix~\ref{app:demandbounds}. Taking into account that induced demand is inversely proportional to prices, the impact of $\sigma$ on the demand function bounds is in accordance with price bounds in \bt{Proposition}~\ref{cor:pricebounds}.


\begin{remark}\label{remark:simultaneousbound}
The upper bound in \eqref{eq:monoduodemandratio} is achieved when $\sigma=1$, $\lambda_{ij}^m=\dbound$ and $\lambda_{ij}^d=0$. Although it is achievable for some OD pairs, it is not possible to achieve it for all OD pairs simultaneously. This is because for $\dijd$ to be 0, constraint \eqref{eq:staticconst1duo} has to be slack, meaning node $i$ has excess supply of vehicles that are being rebalanced to node $j$. This however can not hold simultaneously for all OD pairs, since that would mean there are empty vehicles being routed between all OD pairs, which would not be optimal.
\end{remark}

Interestingly, we see that this potential increase in the aggregate demand never translates into a profit increase for the firms because of the competition. As expected, profits decrease in the presence of competition. According to the next result, the profits generated by a single firm in duopoly is always less than $85\%$ of the profits generated by the monopoly.

\begin{proposition}[Profit Bounds]\label{cor:profitbounds}
\bt{Suppose that Assumptions~\ref{ass:sigma} and \ref{ass:sigmalambda} hold.} Let profits earned by serving the induced demand between OD pair $(i,j)$ in the monopoly be $P^m_{ij}$. With the optimal monopoly prices, $P^m_{ij}$ for all $(i,j)$ obey the following:
\begin{equation}
    \label{eq:monopolyprofitbounds}
   \theta_{ij} {\lmax}/{16}\leq\theta_{ij}\underline{P}^m\leq P^m_{ij}\leq\theta_{ij} \overline{P}^m\leq\theta_{ij}{\lmax}/{4},
\end{equation}
where
\begin{equation*}
    \begin{split}
        \underline{P}^m=\frac{(3\sigma^2+6\sigma-13)^2}{64\sigma(5-3\sigma)^2}\lmax,\quad \overline{P}^m=\frac{(1+\sigma)^2}{16\sigma}\lmax.
    \end{split}
\end{equation*}
Similarly, let profits earned by serving the induced demand between OD pair $(i,j)$ by a single firm in the duopoly be $P^d_{ij}$. With the duopoly \bt{equilibrium} prices, $P^d_{ij}$ for all $(i,j)$ obey:
\begin{equation}
    \label{eq:duopolyprofitbounds}
    0\leq \theta_{ij}\underline{P}^d\leq P^d_{ij}\leq \theta_{ij}\overline{P}^d\leq{(4+\sqrt{10})}\lmax\theta_{ij}/48
\end{equation}
where
\begin{equation*}
    \begin{split}
     \underline{P}^d&{=}\left(\overline{\ell}^d-\dbound\right){\times}\underline{D}^d{=}\frac{1-\sigma}{2\sigma(5-3\sigma)}\lmax,
    \end{split}
\end{equation*}
$$\overline{P}^d=\underline{\ell}^d\overline{D}^d.$$
Furthermore, for all OD pairs, the ratio  $\frac{P^d_{ij}}{P^m_{ij}}$ obeys:
\begin{equation}
    \label{eq:profitsratiomonoduo}
    \frac{8(1-\sigma)}{(\sigma+1)^2(5-3\sigma)}=\frac{\underline{P}^d}{\overline{P}^m}\leq\frac{P^d_{ij}}{P^m_{ij}}\leq\frac{\overline{P}^d}{\underline{P}^m}\lessapprox 0.85.
\end{equation}
\end{proposition}
\bt{\textit{Proof outline: }The proof is done by evaluating the profits for the monopoly given by \eqref{eq:monoprofits} at the price bounds given by \eqref{eq:monopolypricesbound}. For the duopoly, we first \btt{derive} the dual objective, show that it decreases with $\lijd$, and evaluate at the duopoly equilibrium price bounds given by \eqref{eq:duopolypricesbound}. Then, we impose the condition $\sigma\in[3/5,1]$ to get the uniform bounds.}

The complete proof can be found in Appendix~\ref{app:profitbounds}. Since lower prices generate more profits in the monopoly and the price bounds are increasing with $\sigma$, the profit bounds of the monopoly are decreasing with $\sigma$. Similarly, the duopoly profit bounds are decreasing with $\sigma$ too. Since $\sigma$ increases the upper bound on prices, the lower bound on the profits decrease. However, although $\sigma$ decreases the lower bound on the prices, the upper bound on the profits still decrease. This is because competition in the duopoly is a downward driving force on the prices. Consequently, lower prices in the duopoly do not only result from lower $\dijd$, but also stronger competition. Hence, although lower prices increase the aggregate demand, because the firms are now competing over the customers, neither of the firms serve enough customers to compensate for the decrease in the prices. Hence, the profits decrease.


The upper bound in \eqref{eq:profitsratiomonoduo} is achieved when $\sigma=3/5$, $\dijm=\dbound$, and $\dijd=0$. Due to the same argument in Remark~\ref{remark:simultaneousbound}, it can not be achieved simultaneously by all the OD pairs. Consequently, the ratio of total profits can not achieve this upper bound with equality.

How do the customers benefit from the introduction of competition? We saw that a reduction in ride prices is expected. Next, we show that the consumer surplus in the duopolistic setting is at least equal to and can be up to $16$ times the consumer surplus in the monopoly.
\begin{proposition}[Consumer Surplus Bounds]
\label{cor:surplusbounds}
\bt{Suppose that Assumptions~\ref{ass:sigma} and \ref{ass:sigmalambda} hold.} Let the consumer surplus of customers requesting a ride between OD pair $(i,j)$ in the monopoly be $\textnormal{CS}^m_{ij}$. With the optimal monopoly prices, $\textnormal{CS}^m_{ij}$ for all $(i,j)$ obey:
\begin{equation}
    \label{eq:monocsbounds}
    \theta_{ij}\frac{\lmax}{32}\leq\theta_{ij}\underline{\textnormal{CS}}^m\leq\textnormal{CS}^m_{ij}\leq\theta_{ij}\overline{\textnormal{CS}}^m\leq\theta_{ij}\frac{13}{90}\lmax,
\end{equation}
where
$$
\underline{\textnormal{CS}}^m=\frac{171\sigma^4-660\sigma^3+1378\sigma^2-1748\sigma+907}{384\sigma(5-3\sigma)^2}\lmax,
$$
$$
\overline{\textnormal{CS}}^m=({7\sigma^2-2\sigma+7})\lmax/({96\sigma}).
$$
Similarly, let the consumer surplus of customers requesting a ride between OD pair $(i,j)$ in the duopoly be $\textnormal{CS}^d_{ij}$. With the duopoly \bt{equilibrium} prices, $\textnormal{CS}^d_{ij}$ for all $(i,j)$ obey:
\begin{equation}
    \label{eq:duocsbound}
    \theta_{ij}\frac{\lmax}{8}\leq\theta_{ij}\underline{\textnormal{CS}}^d\leq\textnormal{CS}^d_{ij}\leq\theta_{ij}\overline{\textnormal{CS}}^d\leq\theta_{ij}\frac{\lmax}{2},
\end{equation}
where
$$
\underline{\textnormal{CS}}^d=({\sigma^2-2\sigma+13})\lmax/({96\sigma}),
$$
\begin{equation*}
    \begin{split}
    \overline{\textnormal{CS}}^d=\frac{\lmax}{24\sigma(1-\sigma)}\Big(&(2\sigma)^3-(\sigma+1-2\frac{\underline{\ell}^d}{\lmax})^3\\
    &
    -24\sigma(1-\frac{\underline{\ell}^d}{\lmax})(\sigma-1+\frac{\underline{\ell}^d}{\lmax})\Big).
    \end{split}
\end{equation*}

Furthermore, for all OD pairs, the ratio $\frac{\textnormal{CS}^d_{ij}}{\textnormal{CS}^m_{ij}}$ obeys:
\begin{equation}
    \label{eq:csratiobounds}
    1\leq\frac{\sigma^2-2\sigma+13}{7\sigma^2-2\sigma+7}\leq\frac{\underline{\textnormal{CS}}^d}{\overline{\textnormal{CS}}^m}\leq\frac{\textnormal{CS}^d_{ij}}{\textnormal{CS}^m_{ij}}\leq\frac{\overline{\textnormal{CS}}^d}{\underline{\textnormal{CS}}^m}\leq16.
\end{equation}
\end{proposition}
\bt{\textit{Proof outline: }The proof is done by evaluating the consumer surplus for the monopoly given by \eqref{eq:monocs} at the price bounds given by \eqref{eq:monopolypricesbound}. For the duopoly, we compute the consumer surplus at the price bounds given by \eqref{eq:duopolypricesbound} in a similar fashion to the the proof of Proposition~\ref{prop:monoprofitcs}. Then, we impose the condition $\sigma\in[3/5,1]$ to get the uniform bounds.}

The complete proof can be found in Appendix~\ref{app:csbounds}. Considering the fact that lower prices (both in the duopoly and the monopoly) increase the consumer surplus by inducing more customers and increasing the surplus per customer, the dependency of the price bounds on $\sigma$ reflects to the consumer surplus bounds.


Remark~\ref{remark:simultaneousbound} applies for the upper bound in \eqref{eq:csratiobounds} too, and thus it can not be achieved for all OD pairs simultaneously. Therefore, the ratio of total consumer surplus cannot achieve this upper bound with equality.

So far, we have studied the effects of competition in an electric AMoD system by adopting a static network-flow formulation. Although very convenient for analysis, this formulation does not reflect the randomness in arrivals nor constrains vehicles dispatch decisions to be integer valued (e.g., 0.25 customer may be served). To address these discrepancies with the real environment, in the next section, we modify our model to account for the randomness in arrivals and furthermore design a control policy that can be implemented in real-time.

\section{Real-Time Control}\label{sec:realtime}
To accommodate for the stochastic nature of the arrivals, we model the arrival of the potential customers OD pair $(i,j)$ as a Poisson process with an arrival rate of $\theta_{ij}$. Moreover, we allow the firms to set prices real-time and use the same price-responsive demand model. In particular, during period $t$, for a price tuple $(\ell_{ijt}^1,\ell_{ijt}^2)$ for OD pair $(i, j)$, the induced arrival rate for firm $f$ is given by $\Theta_{ijt}^f=\theta_{ij}D(\ell_{ijt}^f,\ell_{ijt}^{-f})$. Thus, the number of new ride requests in time period $t$ for firm $f$ is $A^f_{ijt}\sim\textnormal{Pois}(\Theta_{ijt}^f)$ for OD pair $(i,j)$. As a consequence of this randomness in the customer arrivals, the platform operator might not be able to assign every customer to a ride immediately (if the number of induced arrivals exceed the number of available vehicles). In order to address this nuance, we adopt the following ride-sharing model:

\noindent
\textbf{Ride \bt{Hailing} Model:} Customers that purchase a ride during period $t$ are not immediately matched with a ride, but enter the queue for OD pair $(i,j)$ to be served at the beginning of period $t+1$. After the platform operator executes routing decisions for the fleet at the beginning of period $t+1$, the customers in the queue for OD pair $(i,j)$ are matched with rides and served on a first-come, first-served basis. 

Under these additional modeling modifications, our goal is to establish a real-time pricing and fleet management policy that can be implemented in a real environment \bt{and provides stability of the queues\footnote{\bt{The stability condition that we are interested in is rate stability of all queues. A queue for OD pair $(i,j)$ is rate stable if $\underset{t\rightarrow \infty}{\lim}q_{ij}(t)/t=0$.}}}. In fact, the model studied in Section~\ref{sec:static} is the static planning problem associated with this real environment, where we ignored the stochasticity of the arrivals and used the expected values, while allowing the vehicle routing decisions to be flows (real numbers) rather than integers. For the monopoly (or the symmetric duopoly), the solution to this static planning problem in \eqref{eq:flowoptimization} (or \eqref{eq:duopolyflowoptimization}) is the optimal static policy that consists of optimal prices as well as optimal vehicle routing and charging decisions. This policy can not directly be implemented in a real environment because it does not yield integer valued solutions. In an earlier work \cite{turan2019dynamic}, it was proven that randomizing the vehicle decisions according to the optimal solution of the static problem to get integer-valued actions guarantees stability of the queues. However, considering random arrivals, this method may not execute the most profitable actions since it does not take the real-time queue lengths into consideration. Although it guarantees stability of the queues, it does not seek to minimize the queue lengths and hence the wait time of the passengers, which would negatively affect the business. 

Instead of using the randomized solution to implement real-time actions, it is possible to realize a real-time policy that acknowledges the queue lengths and hence aims to maximize the profits while minimizing the total wait time of the customers. To achieve this, we propose to apply finite-horizon model predictive control (MPC) in our numerical experiment (albeit with no performance guarantee).

\noindent\bt{\textbf{MPC Procedure: }The idea of finite-horizon MPC is to observe the current state of the environment and determine the best control strategy for a planning horizon of $T$ by predicting the state path of the environment. Then, only the control strategy at the initial time period is implemented and the process is repeated. Specifically, let $\cal S$ be the state of the vehicles (locations, energy levels) and $\{Q_{ij}\}_{i,j\in \bt{\cal N}}$ be the outstanding customer demand (i.e., people who have requested a ride but not yet served) at the beginning of planning time. The MPC Algorithm is summarized as follows:}
\begin{algorithm}[h]
	\caption{MPC Procedure}\label{alg:mpc}
	\begin{algorithmic}[1]
        \STATE $\cal{S}\leftarrow$ Get vehicle states (locations, energy levels)
        \STATE $Q_{ij}\leftarrow$ Count outstanding customers
        \STATE $\{x_{ijt}^e,x_{ict}^e,\ell_{ijt}\}_{\forall i,j,e,t}\leftarrow$ Solve \eqref{eq:mpcdymono}
        \STATE Execute $\{x_{ij0}^e,x_{ic0}^e, \ell_{ij0}\}_{\forall i,j,e}$
	\end{algorithmic} 
\end{algorithm}

    \bt{At each period, Algorithm~\ref{alg:mpc} is run and the system state is observed. Using this information, the optimal fleet management and pricing strategy is computed for the next $T$ periods by solving \eqref{eq:mpcdymono}. Vehicle routing/charging and pricing decisions are executed for the initial time period and the environment transitions into next state. Then, Algorithm~\ref{alg:mpc} is re-run and this process is repeated during the entire operation of the system.}
    
\bt{Next, we state the optimization problem \eqref{eq:mpcdymono} for the controller using a dynamic pricing scheme in monopoly. Let the decision variable $\ell_{ijt}^1$ be the price for rides between OD pair $(i,j)$ in period $t$, $x_{ijt}^e$ be the number of vehicles at node $i$ with energy level $e$ being routed to node $j$ in period $t$, $x_{ict}^e$ be the number of vehicles charging at node $i$ starting with energy level $e$ in period $t$, and $q_{ijt}$ be the people waiting in the queue for OD pair $(i,j)$ in period $t$. We state the problem as follows: }
\bt{\begin{subequations}
\label{eq:mpcdymono}
\begin{align}
&\underset{x_{ict}^e,x_{ijt}^e,q_{ijt},\ell_{ijt}^1}{\text{max}}
& &\hspace{-.4cm}\nonumber\sum_{ijt}\ell_{ijt}^1\theta_{ij}D(\ell_{ijt}^1,\infty)-\sum_{ijt}w_{ijt}q_{ijt}\\
& & &\hspace{-.4cm}\label{eq:mpcdymonoobj}-\beta_t\sum_{ijet}\tau_{ij}x_{ijt}^e-\sum_{iet}(\beta_c+p_i)x_{ic}^t\\
& \text{s.t.}
& &\hspace{-1.75cm}\label{eq:mpcdymonoinitialqueue}  q_{ijt_0}\geq Q_{ij}-\sum_{e}x_{ijt_0}^e, \quad \forall i,j\in \bt{\cal N}\\
& & &\hspace{-1.75cm}\nonumber q_{ijt}\geq q_{ijt-1}+\theta_{ij}D(\ell_{ijt-1}^1,\infty)-\sum_{e}x_{ijt}^e,\\
\label{eq:mpcdymonoqueueconstraint}& & &\hspace{1cm} \forall i,j\in \bt{\cal N},~ \forall t>t_0,
\\
& & &\nonumber \hspace{-1.75cm}\sum_{j}x_{ijt}^e+x_{ict}^e-\sum_{j}x_{jit-\tau_{ji}}^{e+e_{ji}}-x_{ict-1}^{e-1}=s_{it}^e,\\
& & &\hspace{1cm}
\label{eq:mpcstaticbalanceconstraint}\forall i\in {\cal N},~\forall e\in {\cal E},~ \forall t \geq t_0\\
& & &\hspace{-1.75cm}x_{ict}^{e_{\max}}=0,\quad\forall i\in\bt{\cal N},~ \forall t\geq t_0,\\
& & &\hspace{-1.75cm}x_{ijt}^e=0,\quad \forall e<e_{ij},\forall i,j\in\bt{\cal N},\forall t\geq t_0,\\
\label{eq:mpcstaticnonnegativity}& & &\hspace{-1.75cm}\nonumber x_{ijt}^e, x_{ict}^e,q_{ijt}\geq 0, ~x_{ijt}^e,x_{ict}^e \in {\mathbb N},\\
& & &\hspace{1cm}\forall i,j\in \bt{\cal N},~ \forall e \in{\cal E},~ \forall t\geq t_0\\
\label{eq:mpcstaticzeroconst}& & &\hspace{-1.75cm} x_{ijt}^e=x_{ict}^e=0, \forall e\notin{\cal E}, \forall t<t_0, \forall i,j\in \bt{\cal N}.
\end{align}
\end{subequations}
}
\bt{The first term in the objective function \eqref{eq:mpcdymonoobj} corresponds to the expected revenue gained by setting prices $\ell_{ijt}^1$. The second term assigns a cost to the queue lengths, where $w_{ijt}$ is the cost per person in the queue for OD pair $(i,j)$ at the time period $t$. The third term is the operational costs of the trip-making vehicles, and the last term is the operational and the charging costs of the charging vehicles. Hence, the objective is to maximize the profits minus the queue penalty.}

\bt{The state variable $s_{it}^e$ denotes the number of vehicles at node $i$ with energy level $e$, at the beginning of time period $t$. At the beginning of the planning time $t=t_0$, $s_{it_0}^e$ is simply the number of available vehicles at node $i$ with energy level $e$. For $t>t_0$, $s_{it}^e$ denotes the number of vehicles that will be available at the beginning of time period t, at node $i$ with energy level $e$. These are the vehicles that are en route to another node at the time of planning. Hence, \eqref{eq:mpcstaticbalanceconstraint} is the vehicle balance constraint. The constraints \eqref{eq:mpcdymonoqueueconstraint} along with the non-negativity constraint \eqref{eq:mpcstaticnonnegativity}, implement the queue length transition $q_{ijt}=\max\{0, q_{ijt-1}+\theta_{ij}D(\ell_{ijt-1}^1,\infty)-\sum_{e}x_{ijt}^e\}$ as two linear inequalities. For $t=t_0$, the queue length is modified via \eqref{eq:mpcdymonoinitialqueue}, where $Q_{ij}$ denotes the number of passengers waiting to be served at the planning time.}

\bt{The MPC controller using a dynamic pricing scheme for the duopoly can be stated in a similar way to the monopoly. Due to space limitations, we exclude it here and refer the reader to the Appendix~\ref{app:mpc}.}

\bt{We end this section by noting that it is possible to implement a model predictive controller with static prices in monopoly simply by adding the constraint $\ell_{ijt}^1=\ell_{ij}^m,\forall t\geq t_0$ to \eqref{eq:mpcdymono}. For the duopoly, we replace $D(\ell_{ijt}^1,\infty)$ with $D(\ell_{ijt}^1,\ell_{ijt}^2)$ and add the constraint $\ell_{ijt}^1=\ell_{ijt}^2=\ell_{ij}^d,\forall t\geq t_0$.}

\section{Numerical Study}\label{sec:numerical}

In this section, we discuss the effects of competition and the performances of the real-time controllers via numerical examples. To solve the optimization problems we used the Gurobi Optimizer \cite{gurobi}. 

In our discrete-time system, we chose one period to be equal to $\Delta t=5$ minutes, which is equal to the time it takes to deliver one unit of battery energy. We chose operational costs of $\beta_t=\$0.2$ and $\beta_c=\$0.1$ (by taking the amortized average price of an electric car  over 5 years \cite{avgevprice} as a reference), maximum willingness to pay  $\ell_{\max}=\$50$, and $\sigma=3/5$. 
We chose a battery capacity of $24$kWh, and discretized the battery energy into $e_{\max}=6$ units, where one unit of battery energy is $4$kWh. Price of electricity per unit of energy (4kWh) ranges from $\$0.32$ to $\$1.2$\cite{electricitycost}, and we randomly sampled $p_i$ for all locations uniformly from this range. 

For the network and demand data, we divided Manhattan into 20 regions as in Figure \ref{fig:manhattanregions}. Using the yellow taxi data from the New York City Taxi and Limousine Commission
dataset \cite{manhattantaxidata} for May 09, 2019, Thursday between 15.00-17.00, we extracted the average arrival rates for rides,\begin{wrapfigure}{r}{0.16\textwidth}
\centering
    \includegraphics[width=.16\textwidth]{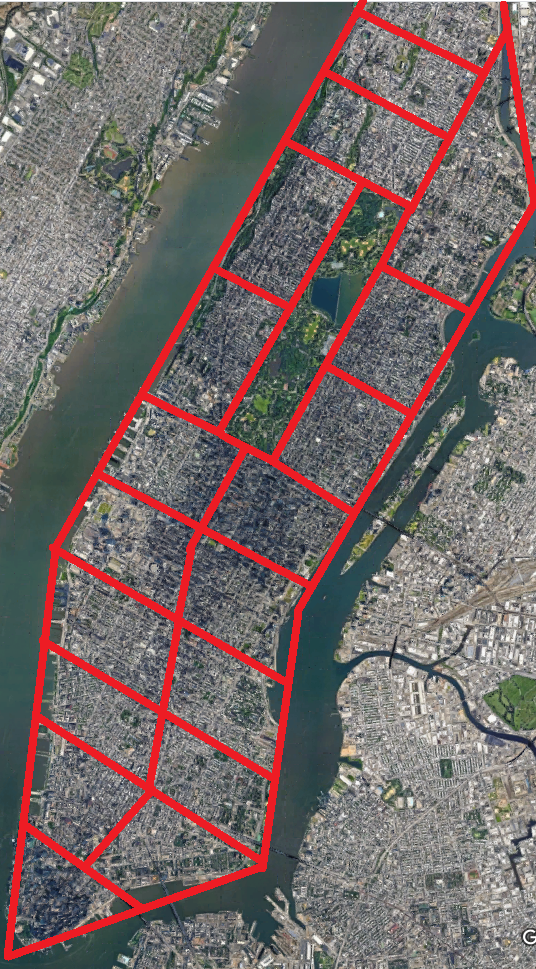}
    \vspace{-.6cm}
  \caption{Manhattan divided into $\bt{n}=20$ regions.}
  \label{fig:manhattanregions}
  \vspace{-.1cm}
\end{wrapfigure} average trip durations, and average distances between the regions (we excluded the rides occurring in the same region).
\bt{Note that the demand data used is not the data of potential riders, but the data of realized rides. Although it is not ideal to impose a demand function on the data of realized rides, this is the best data we could use due to lack of available data on potential riders. This is a common approach in the literature of pricing schemes in ride-sharing platforms \cite{bimpikis2019spatial,wollenstein2020joint}, as the realized rides leaving a location can be seen as a reasonable proxy for the potential riders at that location.}


\vspace{-.1cm}
\subsection{Effects of Competition Under Static Setting}
In this study, we analyze the effects of competition using prices for rides, induced demand, profits, and consumer surplus as metrics. To get the values of the aforementioned metrics in the monopoly, we solved \eqref{eq:flowoptimization}. For the duopoly, \bt{we can not solve \eqref{eq:duopolyflowoptimization} since the problem is non-convex. Therefore, we implemented best-response dynamics to see empirically whether this process would converge to an equilibrium of the duopoly so that we could numerically compare the monopoly and the duopoly.
Although we do not have a theoretical guarantee for convergence of best response dynamics, we know that only symmetric equilibria exist according to Proposition~\ref{prop:equilibrium}. Fortunately, our experiment converged to a symmetric equilibrium in a couple of iterations as demonstrated in Figure~\ref{fig:duopolybestresponse}.}

\begin{figure}[t]
    \centering
    \includegraphics[width=.5\textwidth]{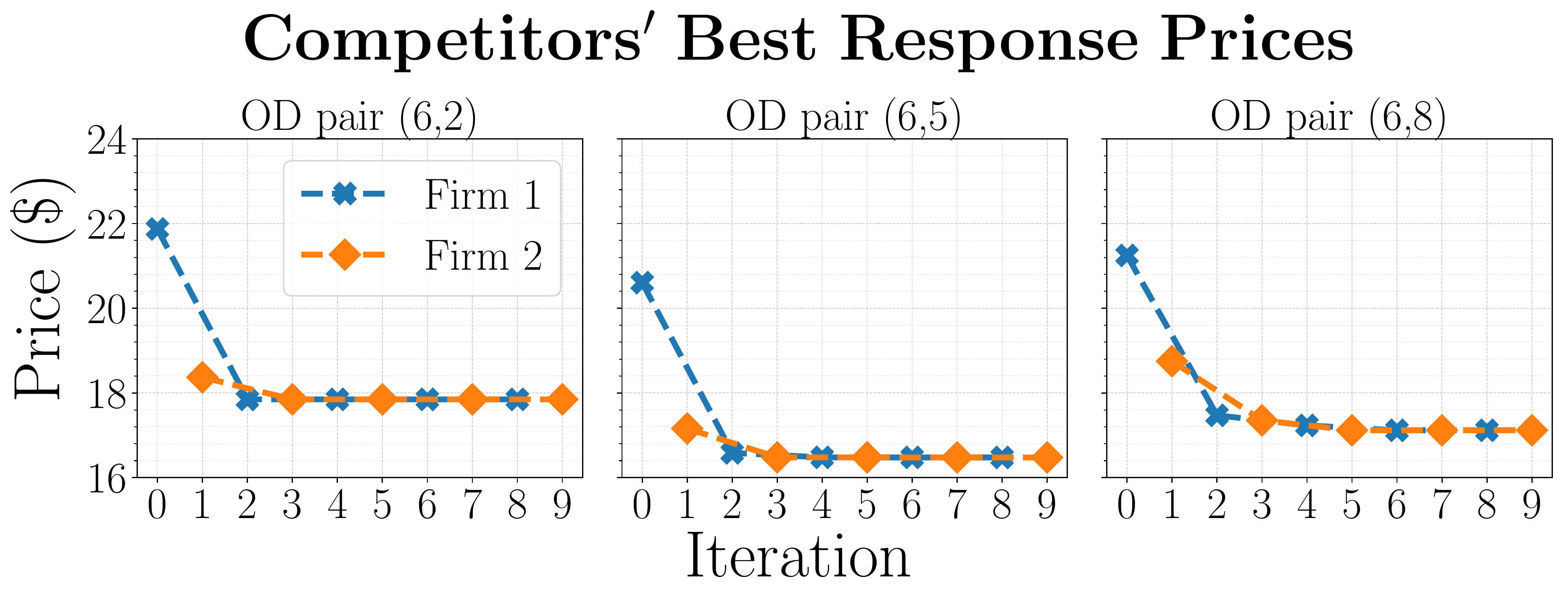}
    \caption{Best response prices for some rides originating from node $6$.}
    \label{fig:duopolybestresponse}
    \vspace{-.3cm}
\end{figure}


In Table~\ref{tab:ratios}, we display the ratios of performance metrics in the monopoly and the symmetric duopoly equilibrium. Moreover, we compute the theoretic upper and lower bounds derived in Section~\ref{sec:static} for $\sigma=3/5$ for comparison. To summarize the table, competition results in a $20\%$ decrease in the average prices of rides, a $44\%$ increase in the total induced demand, a $43\%$ decrease in the profits of a single firm, and a $100\%$ increase in the consumer surplus. 

\begin{table}[h]
\centering
 \begin{tabular}{|c|| c c c |} 
 \hline
  Metrics & Empirical & Theoretic LB & Theoretic UB \\ \hline
  $\ell_{\textnormal{avg}}^d/\ell_{\textnormal{avg}}^m$ &0.80 &0.67 &1 \\ \hline
  $D^d/D^m$ &1.44 &1.25 & 2.26\\ \hline
  $P^d/P^m$ &0.57 &0.39 &0.85 \\ \hline
  $\textnormal{CS}^d/\textnormal{CS}^m$ &2.00 &1.46 &5.89 \\ \hline
 \end{tabular}
\caption{Ratios of average prices, induced demand, profits, and consumer surplus in the monopoly and the symmetric duopoly equilibrium for $\sigma=3/5$.}
\label{tab:ratios}
\end{table}

\noindent
\textit{Impact of $\sigma$}:
The correlation over customers' preferences is measured by $\sigma$, and the effects of competition depend on the value of $\sigma$. To study how $\sigma$ influences the effects of competition, we present the ratios of performance metrics in the monopoly and the symmetric duopoly equilibrium for $\sigma=0.8$ and $\sigma=1$ in Table~\ref{tab:effectofsigma}.

\begin{table}[h]
\centering
 \begin{tabular}{|C{1.cm}||C{.8cm}|C{.8cm}|C{.8cm}|C{.8cm}|C{.8cm}|C{.8cm}|} 
 \hline
\multirow{2}{*}{Metrics} &\multicolumn{2}{c|}{Empirical}  & \multicolumn{2}{c|}{Theoretic LB} & \multicolumn{2}{c|}{Theoretic UB}\\ \cline{2-7}
&$\sigma=0.8$&$\sigma=1$ &$\sigma=0.8$&$\sigma=1$&$\sigma=0.8$&$\sigma=1$\\ \hline
$\ell_{\textnormal{avg}}^d/\ell_{\textnormal{avg}}^m$&0.42 &0.11 &0.29 &0 &1 &1\\\hline
$D^d/D^m$& 1.73&2.04 &1.11 &1 &2.55 &4\\\hline
 $P^d/P^m$&0.32 &0 &0.19 &0 &0.74 &0\\\hline
 $\textnormal{CS}^d/\textnormal{CS}^m$&2.95 &4.18 &1.22 &1 &9.22 &16\\\hline
 \end{tabular}
\caption{Ratios of average prices, induced demand, profits, and consumer surplus in the monopoly and the symmetric duopoly equilibrium for $\sigma=0.8$ and $\sigma=1$.}
\label{tab:effectofsigma}
\vspace{0cm}
\end{table}

The results in Tables~\ref{tab:ratios} and \ref{tab:effectofsigma} indicate that the higher the $\sigma$, the stronger the competition between the firms.  A larger $\sigma$ indicates higher correlation over customers' preferences, which means that the customers care less about the identity of the firm and more about lower prices when buying a ride ($\sigma=1$ means they buy from the firm that offers the lower price). Hence, a stronger competition requires the firms to drop their prices further, which in turn decreases their profits more. This is in favor of the customers, since lower prices induce more demand while generating higher consumer surplus.

\begin{figure*}[t!]
    \centering
    \begin{subfigure}[t]{0.45\textwidth}
        \centering
        \includegraphics[width=\textwidth]{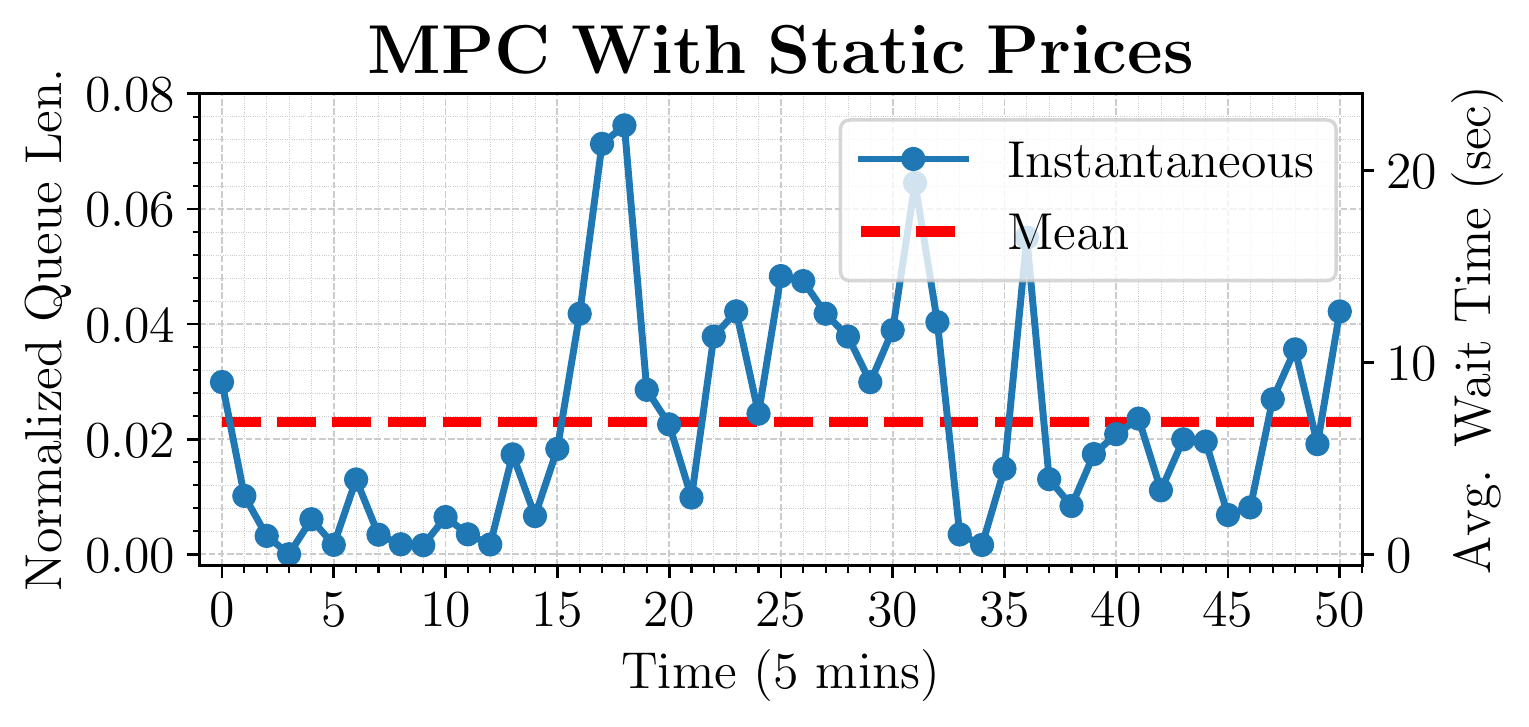}
        \vspace{-.6cm}
        \label{fig:mpcstatic}
    \end{subfigure}%
    \hfill
    \begin{subfigure}[t]{0.45\textwidth}
        \centering
        \hspace*{-.55cm}\includegraphics[width=\textwidth]{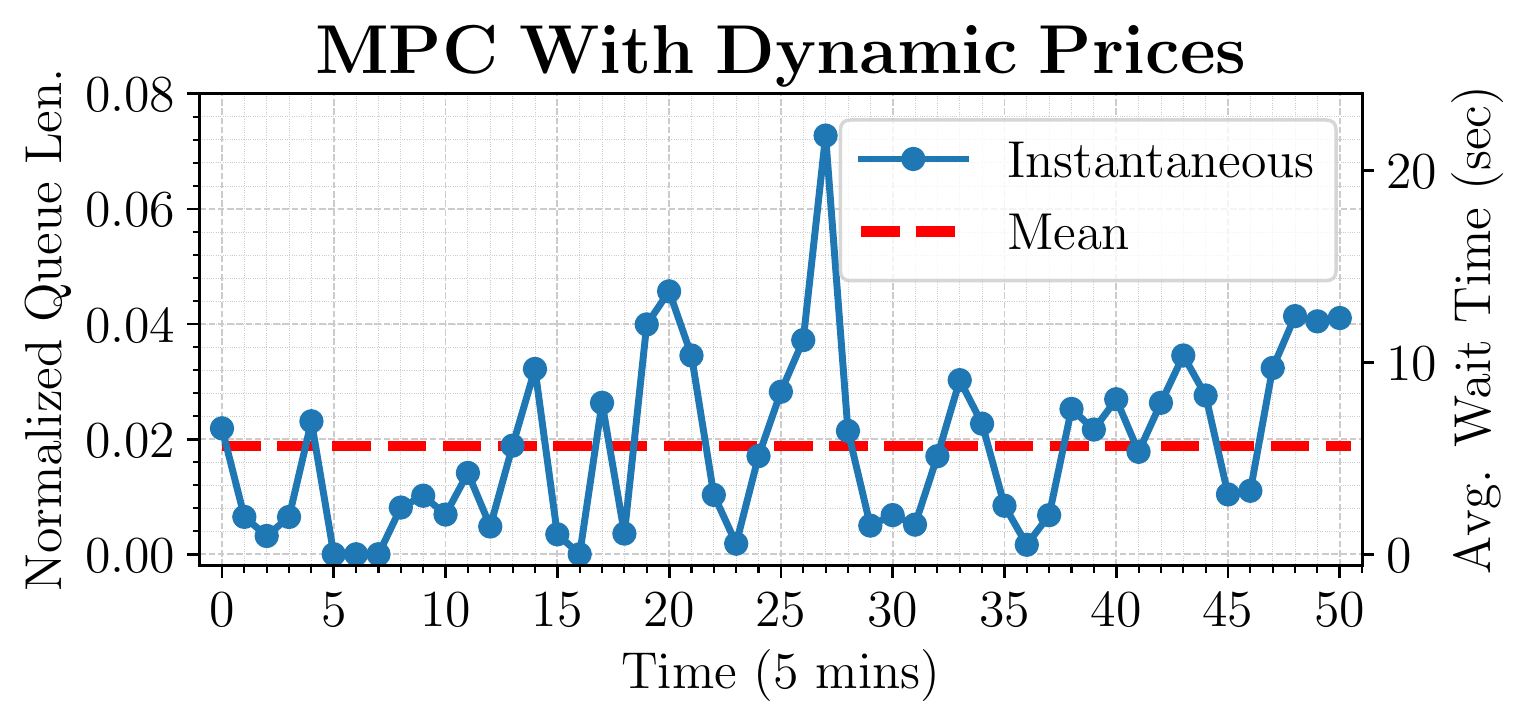}
        \vspace{-.6cm}
        \label{fig:mpcdynamic}
    \end{subfigure}
    \vspace{-.3cm}
    \caption{MPC results. We plot the normalized queue length for the MPC with static prices (left)/MPC with dynamic prices (right).}
    \label{fig:mpcresults}
\end{figure*}

\vspace{-.1cm}
\subsection{Real-Time Control}

In this study, we demonstrate the performances of the model predictive controllers utilizing static and dynamic pricing schemes using profits (minus the queue penalty) and the average wait time of the customers as metrics. To quantify the queue penalty, we set queue penalty per person to be $w_{ijt}=\$4$
(by doubling the average hourly wage of $\$24$ in the U.S.\cite{avgwages}).

\noindent We computed the instantaneous profits in one period as:
\begin{equation}
    \textnormal{Profits}=\textnormal{Revenue}-(\textnormal{Operational + Charging Costs}),
\end{equation}
the queue penalty in one period as:
\begin{equation}
    \textnormal{Queue Penalty}=w\times\textnormal{Outstanding Customers},
\end{equation}
and used the objective value of \eqref{eq:flowoptimization} as an upper bound on the average profits for comparison. We define
\begin{equation}
    \textnormal{Normalized Queue Length}\eqdef\frac{\textnormal{Outstanding Customers}}{\textnormal{Induced Demand}}
\end{equation}
and compute the instantaneous average wait time of customers in one period as:
\begin{equation}
    \textnormal{Avg. Wait Time}=\textnormal{Normalized Queue Length}\times\Delta t.
\end{equation}

We implemented the MPC with $T=10\times\Delta t$ as the planning horizon, and ran the environment for $50\times \Delta t$.

\subsubsection{Monopoly} 

We plot the instantaneous average wait time for MPC with static prices (MPC-SP) and dynamic prices (MPC-DP)
in Figure~\ref{fig:mpcresults}, and summarize the results in Table~\ref{tab:mpcresults}.

\begin{figure}[t]
    \centering
        \vspace{-.3cm}
    \includegraphics[width=.45\textwidth]{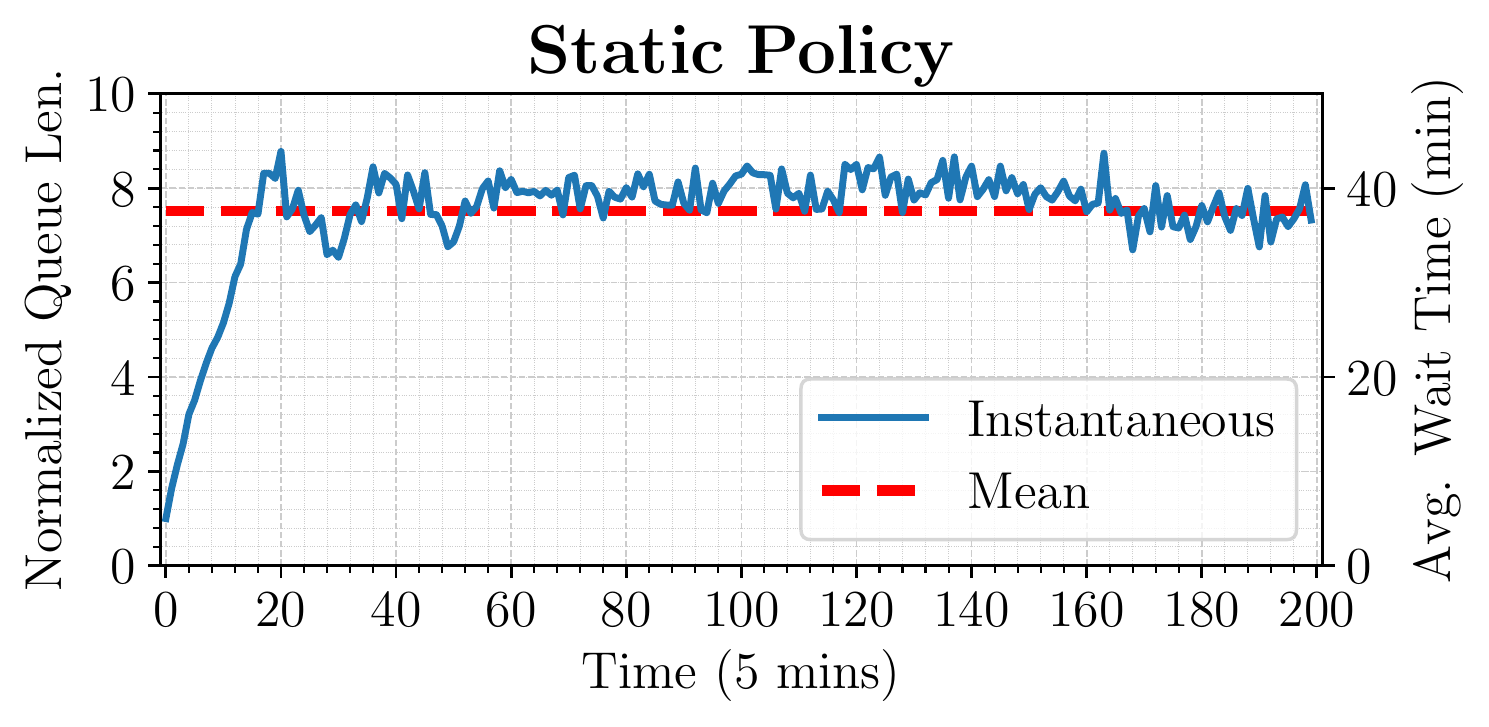}
    \vspace{-.3cm}
    \caption{Monopoly Static Policy Queues}
    \label{fig:monopolystaticqueues}
\end{figure}

\begin{table}[h]
\centering
 \begin{tabular}{|c|| c c |c|} 
 \hline
  Metrics &  MPC-SP & MPC-DP & \% Impr. \\ [0.5ex] 
 \hline\hline
 Mean Profits-Queue Penalty (\$) & 11700.86 & 11778.13 & \multirow{2}{*}{0.66\%}\\ \cline{1-3}
  \% of static &  98.36\% & 99.02\% & \\\hline\hline
 Mean Avg. Wait Time (sec) & 6.91 & 5.64 &  18.38\%\\
 \hline
 Var. Avg. Wait Time (sec) & 32.58 & 20.95 & 35.7\%\\\hline 
 
 \end{tabular}
\caption{MPC results in the monopoly. Mean and variance are computed over time. The static objective value is $11894.9$.}
\label{tab:mpcresults}
\end{table}

We observe that both controllers are able to keep the queue lengths very short (around $2\%$ of the induced demand), and still generate substantial amount of profits that is close to the static objective. In particular, MPC-SP generates $98.36\%$ and MPC-DP generates $99.02\%$ of the static profits, including the queue penalty. Although the marginal benefits of using dynamic pricing might seem low, a $0.66\%$ increase in average profits would make a considerable difference in the long run (e.g., from Table~\ref{tab:mpcresults}, a $\bt{\$}77$ increase in profits per period adds up to more than an increase of $\bt{\$}900$ per hour). Moreover, we observe that the mean of average wait time for MPC-SP is 6.91 seconds, while that of MPC-DP is 5.64 seconds which is an improvement of $18.38\%$. Lastly, a dynamic pricing scheme reduces the variance of the average wait time by $35.7\%$, which indicates a more robust system with predictable wait times.

We furthermore generated integer actions by randomizing according to the flows of the static solution and implemented the static policy in the real environment to compare its performance. In Figure~\ref{fig:monopolystaticqueues} we plot the average wait time using the static policy. Although it provides stability of the queues, it results in bad wait times with a mean of 36.9 minutes, which is more than 300 times longer than both MPC-SP and MPC-DP.

\subsubsection{Duopoly}
We computed the mean value of the metrics over both firms to get the performances of the controllers. The results are summarized in Table~\ref{tab:mpcresultsduopoly}. 

\begin{table}[h]
\centering
 \begin{tabular}{|c|| c c |c|} 
 \hline
  Metrics &  MPC-SP & MPC-DP & \% Impr. \\ [0.5ex] 
 \hline\hline
 Mean Profits-Queue Penalty (\$) & 6670.89 & 6729.2 & \multirow{2}{*}{0.87\%}\\ \cline{1-3}
  \% of static &  98.56\% & 99.42\% & \\\hline\hline
 Mean Avg. Wait Time (sec) & 7.27 & 5.01 & 31.08\%\\
 \hline
 Var. Avg. Wait Time (sec) & 44.68 & 17.92 & 59.89\%\\\hline 
 
 \end{tabular}
\caption{MPC results in the duopoly. Mean and variance are computed over time. The static objective value is $6768.2$.}
\label{tab:mpcresultsduopoly}
\end{table}

Similar to the monopoly, both controllers are able to keep the queues short while generating profits close to the static objective, with dynamic pricing scheme increasing the efficiency. 

 

\section{Conclusion}\label{sec:conclusion}
In this paper, we studied the impacts of competition on electric AMoD systems by comparing the monopoly and the duopoly in equilibrium. By formalizing the optimal strategies of profit-maximizing platform operators, we show that the identical competitors can only be in a symmetric equilibrium. In state of a symmetric duopoly equilibrium, the prices for rides and the profits of the firms are always less than those in the monopolistic setting, whereas the aggregate demand served and the consumer surplus are always higher. The closed-form universal bounds we provide quantify the amount of increase/reduction on the said metrics. These bounds depend heavily on correlation between customers' preferences and therefore the strength of the competition. The numerical studies using network and demand data of Manhattan indicate that stronger competition boosts the amount of increase/reduction on the metrics. Lastly, we experimentally demonstrate that it is possible to implement a real-time control policy for fleet management using model predictive control, and show that a real-time pricing policy further improves the performance.


\bibliographystyle{IEEEtran}
\bibliography{references}

\vspace{-1.2cm}
\begin{IEEEbiography}[{\includegraphics[width=1in,height=1.21in,clip,keepaspectratio]{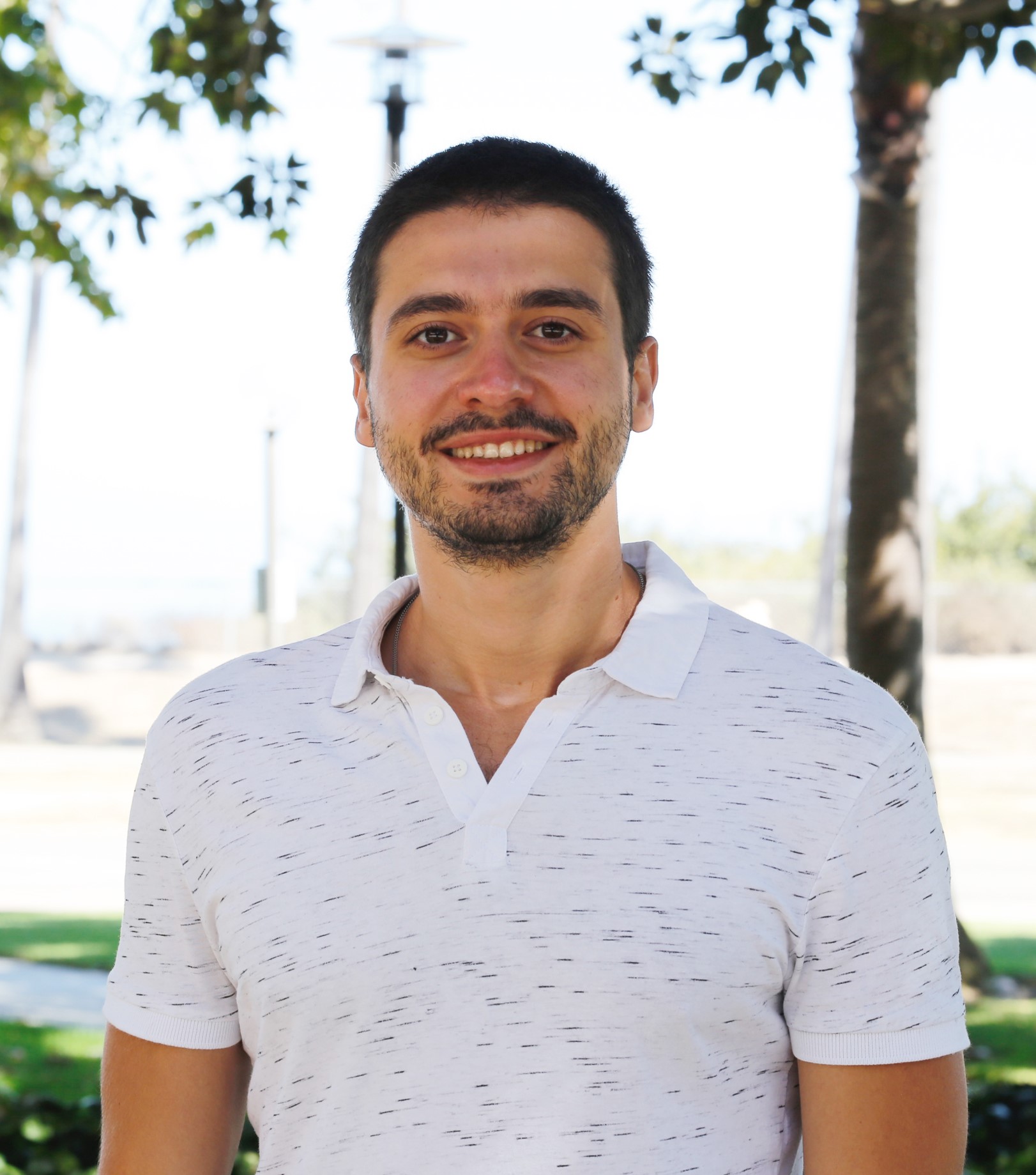}}]{BERKAY TURAN}is pursuing the Ph.D. degree in Electrical and Computer Engineering at the University of California, Santa Barbara. He received the B.Sc. degree in Electrical and Electronics Engineering and the B.Sc. degree in Physics from \ Bo\u gazi\c ci University, Istanbul, Turkey, in 2018. His research interests include optimization and learning for the design, control, and analysis of  multi-agent cyber-physical systems.
\end{IEEEbiography}
\vspace{-1.2cm}
\begin{IEEEbiography}[{\includegraphics[width=1in,height=1.21in,clip,keepaspectratio]{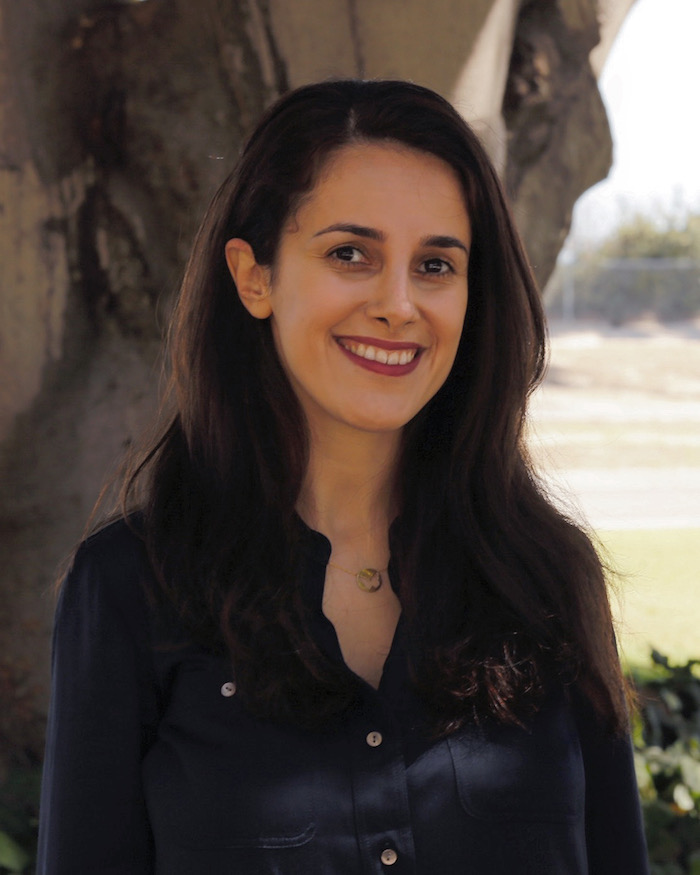}}]{MAHNOOSH ALIZADEH} is an assistant professor of Electrical and Computer Engineering at the University of California Santa Barbara. 
She received the B.Sc. degree (’09) in Electrical Engineering from Sharif University of Technology and the M.Sc. (’13) and Ph.D. (’14) degrees in Electrical and Computer Engineering from the University of California Davis. From 2014 to 2016, she was a postdoctoral scholar at Stanford University. Her research interests are focused on designing network control, optimization, and learning frameworks to promote efficiency and resiliency in societal-scale cyber-physical systems. Dr. Alizadeh is a recipient of the NSF CAREER award.
\end{IEEEbiography}

\appendix

\subsection{Proof of Proposition \ref{prop:monomarginalprices}}\label{sec:proofprop1}
For brevity of notation, let $\beta_c+p_i=P_i$. Let $\lambda_{ij}$ be the dual variables corresponding to the demand satisfaction constraints and $\mu^e_i$ be the dual variables corresponding to the flow balance constraints.
We can \btt{state} the dual problem as:
\begin{subequations}\label{eq:dualproblem}
\begin{align}
\label{eq:dualobjective}
    &\underset{\lambda_{ij},\mu_i^e}{\text{min}}\underset{\ell_{ij}^1}{\text{max}}
    & &\sum_{i=1}^\bt{n}\sum_{j=1}^\bt{n}\theta_{ij}D(\lija,\infty)\left(\ell_{ij}^1-\lambda_{ij}\right)\\
    & \text{subject to}
\label{eq:dualconstraint1}    & & \lambda_{ij}\geq 0,\\
 \label{eq:dualconstraint3}  & & & \lambda_{ij}+\mu_i^e-\mu_j^{e-e_{ij}}-\beta_t \tau_{ij}\leq 0,\\
\label{eq:dualconstraint4}    & & & \mu_i^e-\mu_i^{e+1}-P_i \leq  0\quad \forall i,j,e.
\end{align}
\end{subequations}
For fixed $\lambda_{ij}$
and $\mu_i^e$, the first order optimality condition is:
\begin{equation}
    \label{eq:firstorderopt}
    \frac{\partial D(\lija,\infty)}{\partial \lija}(\lija-\lambda_{ij})+ D(\lija,\infty)=0
\end{equation}
Depending on the region $\lija$ is in, the demand function $D(\lija,\infty)$ has different forms:
\begin{equation}
\label{eq:demandfunctionsmonopoly}
   \hspace{-.2cm} D(\lija,\infty){=}\left\{
      \begin{array}{ll}
      1{-}\frac{(\lija)^2}{2\lmax^2\sigma(1{-}\sigma)} &, \frac{\lija}{\lmax}<(1{-}\sigma) \\[1ex]
      \frac{1+\sigma-\frac{2\lija}{\lmax}}{2\sigma} &, (1-\sigma)\leq\frac{\lija}{\lmax} < \sigma \\[1ex]
      \frac{(1-\frac{\lija}{\lmax})^2}{2\sigma(1-\sigma)} & ,\sigma\leq \frac{\lija}{\lmax}\leq 1 \\
\end{array} 
\right. 
\end{equation}
First, suppose that $\frac{\lija}{\lmax}<(1-\sigma)$. Solving for $\lija$ in \eqref{eq:firstorderopt} using \eqref{eq:demandfunctionsmonopoly}, we get:
\begin{equation}
\label{eq:optimalpricesproof1}
    \lijm=\left({\dij+\sqrt{\dij^2+6\lmax^2\sigma(1-\sigma)}}\right)/{3}.
\end{equation}
Furthermore, the second order optimality condition satisfies:
\begin{equation}
    \label{eq:secondorderopt}
    \frac{\partial^2D(\lija,\infty)}{\partial (\lija)^2}(\lijm-\lambda_{ij})+2\frac{\partial D(\lija,\infty)}{\partial \lija}\Bigg\rvert_{\lija=\lijm}<0.
\end{equation}
Hence, KKT conditions are satisfied and the optimal primal solution satisfies the dual solution with optimal dual variables $\dijm$. By checking the condition $\lijm\leq(1-\sigma)\lmax$ using \eqref{eq:optimalpricesproof1}, we get the condition that $\dijm\leq\frac{3-5\sigma}{2}$.
The optimal prices for the regions where $\frac{\lijm}{\lmax}\in[1-\sigma,\sigma)$ and $\frac{\lijm}{\lmax}\in[\sigma,1]$ are derived in a similar fashion using the demand functions in those regions given in \eqref{eq:demandfunctionsmonopoly}.

To get the upper bound on prices, we go through the following algebraic calculations using the constraints. The inequality \eqref{eq:dualconstraint4} gives:
\begin{equation}
    \label{eq:dualconstraint5}
        \mu_i^{e-e_{ji}}\leq e_{ji}P_i+\mu_i^e,
\end{equation}
and equivalently:
\begin{equation}
    \label{eq:dualconstraint6}
    \mu_j^{e-e_{ij}}\leq e_{ij}P_j+\mu_j^e.
\end{equation}
The inequalities \eqref{eq:dualconstraint3} and \eqref{eq:dualconstraint1} yield:
\begin{equation*}
    \label{eq:dualconstraint7}
    \mu_i^e-\mu_j^{e-e_{ij}}-\beta_t \tau_{ij}\leq 0,
\end{equation*}
and equivalently:
\begin{equation}
    \label{eq:dualconstraint8}
    \mu_j^e-\mu_i^{e-e_{ji}}-\beta_t \tau_{ji}\leq 0,
\end{equation}
Inequalities \eqref{eq:dualconstraint5} and \eqref{eq:dualconstraint8}:
\begin{equation}
\label{eq:dualconstraint9}
\mu_j^e\leq \mu_i^e+\beta_t\tau_{ji}+e_{ji}P_i.
\end{equation}
And finally, the constraint \eqref{eq:dualconstraint3}:
\begin{align*}
    \lambda_{ij}\leq\beta \tau_{ij}+\mu_j^{e-e_{ij}}-\mu_i^e&\overset{\eqref{eq:dualconstraint6}}{\leq}\beta_t \tau_{ij}+e_{ij}P_j+\mu_j^e-\mu_i^e\\
    &\overset{\eqref{eq:dualconstraint9}}{\leq}\beta_t
    \tau_{ij}+e_{ij}P_j+\beta \tau_{ji}+e_{ji}P_i.
\end{align*}
Replacing $P_i=p_i+\beta_c$ and rearranging the terms:
\begin{align}
    \label{eq:lambdaupperbound}
    \lambda_{ij}{\leq} \beta_t(\tau_{ij}{+}\tau_{ji}){+}e_{ij}(p_j{+}\beta_c){+}e_{ji}(p_i{+}\beta_c){=}\overline{\lambda}_{ij},
\end{align}
where the last equality follows from the definition provided in the proposition. Hence, we get the desired upper bound on the prices using the upper bound on the dual variables.
\subsection{Proof of Proposition~\ref{prop:monoprofitcs}}\label{app:profcsmono}
 Using Assumption~\ref{ass:sigmalambda}, we see that $\frac{(3-5\sigma)}{2}\leq0$ and $\frac{(3-5\sigma)}{2}\lmax\leq\dijm\leq\underset{i,j}{\max} ~\overline{\lambda}_{ij}\leq\frac{(3\sigma-1)(3-\sigma)}{4(5-3\sigma)}\ell_{\max}\leq\frac{3\sigma-1}{2}\lmax$. Hence, the optimal prices fall in the region $[(1-\sigma)\lmax,\sigma\lmax)$, and are given by:
\begin{equation}
    \label{eq:optpriceslinear}
    \lijm=({(1+\sigma)\lmax+\dijm})/{4}.
\end{equation}
 The dual problem with optimal prices in \eqref{eq:optpriceslinear} can be \btt{stated} as:
\begin{subequations}
    \label{eq:dualwithoptimalprices}
    \begin{align}\label{eq:dualobjectivewithoptimalprices}
    &\underset{\lambda_{ij}, \mu_i^e}{\text{min}}
    & &\sum_{i=1}^\bt{n}\sum_{j=1}^\bt{n}\frac{\theta_{ij}}{4\sigma\ell_{\max}}\left(\frac{(1+\sigma)\lmax-2\lambda_{ij}}{2}\right)^2\\
    & \text{subject to}
    & & \eqref{eq:dualconstraint1}-\eqref{eq:dualconstraint4}.
    \end{align}
\end{subequations}
The objective function in \eqref{eq:dualobjectivewithoptimalprices} with optimal dual variables, along with \eqref{eq:optpriceslinear} suggests:
\begin{equation*}
    P^m=\sum_{i=1}^\bt{n}\sum_{j=1}^\bt{n}\frac{\theta_{ij}}{4\sigma\ell_{\max}}((1+\sigma)\ell_{\max}-2\lijm)^2,
\end{equation*}
where profits $P^m$ is the optimal value of the objective function of both primal and dual problems (Since the demand function is linear in the specified region, the problem is convex and KKT conditions are satisfied. Hence, strong duality holds).

 Consumer surplus is given by the difference between the price that customers pay and the price that they are willing to pay. For OD pair $(i,j)$ the customers with $v_1>\lijm$ have a positive surplus of $v_1-\lijm$ and the customers with $v_1\leq\lijm$ have a zero surplus since they either do not take the ride or have exactly a valuation of $\lijm$. Since $v_1=\sigma x+(1-\sigma)y$ and $x$ and $y$ are iid uniform random variables in $[0,\lmax]$, the consumer surplus for a single unit of potential \bt{riders} between OD pairs $(i,j)$ is computed as:
\begin{equation}
    \begin{split}
        &\int_0^{\lmax}\int_{\frac{\lijm-(1-\sigma)y}{\sigma}}^{\lmax}\frac{1}{\lmax^2}(\sigma x+(1-\sigma)y-\lijm)dxdy=\\
        &\frac{\lmax(\sigma^2+\sigma+1)-3\lijm(1+\sigma-\frac{\lijm}{\lmax})}{6\sigma}.
    \end{split}
\end{equation}
The total consumer surplus is then:
\begin{equation}
        \label{eq:csmonoproof}
    \textnormal{CS}^m\sum_{i=1}^\bt{n}\sum_{j=1}^\bt{n}\theta_{ij}\frac{\lmax(\sigma^2+\sigma+1)-3\lijm(1+\sigma-\frac{\lijm}{\lmax})}{6\sigma}.
\end{equation}
\subsection{Proof of Proposition~\ref{prop:unique}}\label{app:unique}
In order to prove that the firms are in an equilibrium, we first follow similar steps as the proof of Proposition \ref{prop:monomarginalprices} and determine the optimal prices. Suppose that $\lija,\;\lijb\in\left[\frac{1-\sigma}{2}\lmax,(1-\sigma)\lmax\right]$. In that region:
\begin{equation}
\label{eq:demandduo}
    D(\lija,\lijb)=\frac{4\sigma(1-\sigma-\frac{\lija-\lijb}{\lmax})-(\frac{\lija+\lijb}{\lmax}+\sigma-1)^2}{8\sigma(1-\sigma)}
\end{equation}
\begin{equation}
    \label{eq:demandderivduo}
    \frac{\partial D(\lija,\lijb)}{\partial\lija}=\frac{1}{\lmax}\frac{-6\sigma-\frac{2\lija}{\lmax}-\frac{2\lijb}{\lmax}+2}{8\sigma(1-\sigma)}
\end{equation}
\begin{equation}
    \label{eq:demandsecderivduo}
    \frac{\partial^2D(\lija,\lijb)}{\partial(\lija)^2}=\frac{1}{\lmax^2}\frac{-2}{8\sigma(1-\sigma)}
\end{equation}
Evaluated at $\lija=\lijb$, the above expressions become:
\begin{equation}
\label{eq:demandduoequal}
    D(\lija,\lijb)\Big\rvert_{\lija=\lijb}=\frac{1}{2}-\frac{(\frac{2\lija}{\lmax}+\sigma-1)^2}{8\sigma(1-\sigma)}
\end{equation}
\begin{equation}
    \label{eq:demanddderivduoequal}
    \frac{\partial D(\lija,\lijb)}{\partial\lija}\Bigg\rvert_{\lija=\lijb}=\frac{1}{\lmax}\frac{-6\sigma-\frac{4\lija}{\lmax}+2}{8\sigma(1-\sigma)}
\end{equation}
For a given $\lambda_{ij}$, the first order optimality condition is:
\begin{equation}
    \label{eq:firstoptduo}
    \frac{\partial D(\lija,\lijb)}{\partial\lija}\Bigg\rvert_{\lija=\lijb}(\lija-\lambda_{ij})+ D(\lija,\lijb)\Big\rvert_{\lija=\lijb}=0.
\end{equation}
We plug equations \eqref{eq:demandduoequal} and \eqref{eq:demanddderivduoequal} into the above expression to get a quadratic equation in $\lija$, which has two solutions. One of the solutions is infeasible with $\lija<0$. Hence, we get a unique solution at:
\begin{equation}\label{eq:optpricesduoproof}
\ell_{ij}^{d}=\frac{(3-5\sigma)\lmax+2\lambda_{ij}+\sqrt{\Delta_1}}{8},    
\end{equation}
where $\Delta_1=4\lmax^2+(2\lambda_{ij}+(15\sigma-3)\ell_{\max})(2\lambda_{ij}+(1-\sigma)\ell_{\max})$. Note that in the region where $\lija=\lijb\leq(1-\sigma)\lmax$, $D(\lija,\lijb)$ is concave and thus we need to check the second order optimality condition:
\begin{equation}
    \label{eq:secondoptduo}
    \frac{\partial^2D(\lija,\lijb)}{\partial(\lija)^2}(\lijd-\lambda_{ij})+2 \frac{\partial D(\lija,\lijb)}{\partial\lija}\Bigg\rvert_{\lija=\lijb=\lijd}<0.
\end{equation}
By plugging Equations \eqref{eq:demandsecderivduo}, \eqref{eq:demanddderivduoequal}, and \eqref{eq:optpricesduoproof} into the above expression, one verifies that it holds true. Hence, KKT conditions are satisfied and the optimal primal solution satisfies the dual solution with optimal dual variables $\lambda_{ij}^{d}$:
\begin{equation}\label{eq:optduoprices1}
  \ell_{ij}^{d}=\frac{(3-5\sigma)\lmax+2\lambda_{ij}^d+\sqrt{\Delta_1^*}}{8},
\end{equation}
where $\Delta_1^*=4\lmax^2+(2\lambda_{ij}^d+(15\sigma-3)\ell_{\max})(2\lambda_{ij}^d+(1-\sigma)\ell_{\max})$.
Since the conjecture was that $\lijd\in\left[\frac{1-\sigma}{2}\lmax,(1-\sigma)\lmax\right]$, we check:
$$
\frac{1-\sigma}{2}\lmax\leq\frac{(3-5\sigma)\lmax+2\dijd+\sqrt{\Delta_1^*}}{8}\leq(1-\sigma)\lmax,
$$
to get $\dijd\leq\frac{3(1-\sigma)^2}{2(1+\sigma)}$. For $\dijd=0$, \eqref{eq:optduoprices1} evaluates to $\frac{(3-5\sigma)+\sqrt{-15\sigma^2+18\sigma+1}}{8}\lmax\geq\frac{1-\sigma}{2}\lmax$, hence the prices fall in the specified region.

Now suppose that $\lija,\lijb\in((1-\sigma)\lmax,\frac{\sigma+1}{2}\lmax]$. In that region:
\begin{equation}
\label{eq:demandduo2}
    D(\lija,\lijb)=\frac{(1-\sigma+\frac{\lijb-\lija}{\lmax})(3+\sigma-\frac{3\lija+\lijb}{\lmax})}{8\sigma(1-\sigma)}
\end{equation}
By following the same steps as before, we get optimal prices uniquely as:
\begin{equation}
\label{eq:optduoprices2}
    \lijd=\frac{(5-3\sigma)\lmax+2\lambda_{ij}^d-\sqrt{\Delta_2^*}}{4},
\end{equation}
where $\Delta_2^*=2(\sigma\lmax-\dijd)^2+2(\lmax-\dijd)^2+11(\sigma-1)^2\lmax^2$. By imposing the condition that $\lijd\in((1-\sigma)\lmax,\frac{\sigma+1}{2}\lmax]$, one identifies:
\begin{equation}
    \frac{3(1-\sigma)^2}{2(1+\sigma)}\lmax<\dijd\leq\frac{3\sigma+1}{4}\lmax.
\end{equation}

The upper bound on the dual variables is derived identically to the Proposition \ref{prop:monomarginalprices}. Hence according to Assumption~\ref{ass:sigmalambda} $$\dijd\leq\overline{\lambda}_{ij}\leq\frac{(3\sigma-1)(3-\sigma)}{4(5-3\sigma)}\ell_{\max}<\frac{3\sigma+1}{2}\lmax,$$
is satisfied.

All in all, we get the optimal prices as:
\begin{equation}
    \label{eq:optduoprices}
    \lijd=\left\{
\begin{array}{ll}
      \frac{(3-5\sigma)\lmax+2\dijd+\sqrt{\Delta_1^*}}{8} & \frac{\dijd}{\lmax}\leq\frac{3(\sigma-1)^2}{2(\sigma+1)} \\
      \frac{(5-3\sigma)\lmax+2\lambda_{ij}^d-\sqrt{\Delta_2^*}}{4} & o.w. \\

\end{array} 
\right.
\end{equation}


We now show that when both firms set prices equal to $\{\ell_{ij}^d\}_{i,j\in\bt{\cal N}}$, they are in an equilibrium. Given firm $-f$'s prices equal to $\{\ell_{ij}^d\}_{i,j\in\bt{\cal N}}$, firm $f$ solves the following best response problem to determine its optimal prices:

\begin{subequations}
\label{eq:bestresponse}
\begin{align}
\nonumber&\underset{x_{ic}^e,x_{ij}^e,\ell_{ij}^{f}}{\text{max}}
\nonumber& &\sum_{i=1}^\bt{n}\sum_{j=1}^\bt{n}\theta_{ij}\ell_{ij}^fD(\ell_{ij}^f,\lijd)\\
\nonumber& & &-\sum_{i=1}^\bt{n}\sum_{e=0}^{e_{\max}-1} (\beta_c+p_i) x_{ic}^e\\
& & &-\beta_t \sum_{i=1}^\bt{n}\sum_{j=1}^\bt{n}\sum_{e=e_{ij}}^{e_{\max}} x_{ij}^e\tau_{ij} \\
& \text{subject to}
&  &\theta_{ij}D(\ell_{ij}^f,\lijd) \leq \sum_{e=e_{ij}}^{e_{\max}}x_{ij}^e \quad\forall i,j \in \bt{\cal N},\\
\nonumber& & &\eqref{eq:staticconst2}-\eqref{eq:staticconst8}.
\end{align}
\end{subequations}

The first order optimality condition states:
\begin{equation}
    \frac{\partial D(\ell_{ij}^f,\lijd)}{\partial\ell_{ij}^f}(\ell_{ij}^f-\lambda_{ij})+ D(\ell_{ij}^f,\lijd)=0.
\end{equation}

Setting $\ell_{ij}^f=\lijd$ satisfies the above equation with the optimal dual variable $\dijd$ because $\ell_{ij}^1=\ell_{ij}^2=\lijd$ is a solution to \eqref{eq:firstoptduo} with $\lambda_{ij}=\lambda_{ij}^d$. Since both firms have the identical best response problem \eqref{eq:bestresponse}, the first order condition is satisfied for both when $\ell_{ij}^1=\ell_{ij}^2=\lijd,\forall i,j\in\bt{\cal N}$, and hence the firms are in an equilibrium.

\subsection{Proof Of Proposition~\ref{prop:equilibrium}}\label{app:equilibrium}
We show that when $\lija\neq \lijb$ and both firms serve greater than zero demand  for an OD pair $(i,j)$, the firms cannot be in equilibrium. We do it by showing that the first order condition can not hold for both firms simultaneously.

We let $\lija=\lijb+\delta\lmax$, and add the following constraints:
\begin{itemize}
    \item We constrain $\delta<(1-\sigma)$ (If $\delta\geq(1-\sigma)$, then firm 1 does not serve any demand for that OD pair since the lines depicted on Figure~\ref{fig:duopolydemand} intersect at $y\geq\lmax$).
    \item We let $\lija+\lijb\geq(1-\sigma)\lmax$ (Otherwise if $\lija+\lijb=(1-\sigma)\lmax-2\epsilon$ lines depicted in Figure~\ref{fig:duopolydemand} intersect at $x=\frac{-2\epsilon}{2\sigma}$. Then both firms can increase their profits by increasing their prices by $\epsilon$, while keeping the demand same).
    \item We let $\lija+\lijb\leq(1+\sigma)\lmax$ (Otherwise, the lines depicted in Figure~\ref{fig:duopolydemand} intersect at $x\geq\lmax$, and hence their prices don't affect each others' demand. In that case, the prices are determined according to the monopoly prices, which are bounded by $\sigma\lmax$ according to Assumption~\ref{ass:sigmalambda} and hence their sum is always bounded by $(1+\sigma)\lmax$).
\end{itemize}
Depending on whether $\lija$ and $\lijb$ are greater than $(1-\sigma)\lmax$, we have different demand functions and hence we study the following three cases: 

\vspace{.2cm}
\textbf{Case 1:} Let $\lija,\lijb\leq(1-\sigma)\lmax$. For ease of notation, we define $\ell_f\eqdef\frac{\ell_{ij}^f}{\lmax}$ for firm $f$. When $\la,\lb\leq 1-\sigma$, the demand function for firm $f$ is given by:

\begin{equation}
\label{eq:demandduo0}
    D(\ell_{f},\ell_{-f})=\frac{4\sigma(1-\sigma-(\ell_f-\ell_{-f}))-(\ell_f+\ell_{-f}+\sigma-1)^2}{8\sigma(1-\sigma)}
\end{equation}
\begin{equation}
    \label{eq:demandderivduo0}
    \frac{\partial D(\ell_f,\ell_{-f})}{\partial\ell_f}=\frac{-6\sigma-2\ell_f-2\ell_{-f}+2}{8\sigma(1-\sigma)}
\end{equation}
Using \eqref{eq:demandduo0} and $\la-\lb=\delta$, the demand functions \btt{are determined} as:
\begin{equation}
\label{eq:demand1case1}
    D(\la,\lb)=\frac{4\sigma(1-\sigma-\delta)-(\sigma+2\la-1-\delta)^2}{8\sigma(1-\sigma)},
\end{equation}
\begin{equation}
\label{eq:demand2case1}
    D(\lb,\la)=\frac{4\sigma(1-\sigma+\delta)-(\sigma+2\la-1-\delta)^2}{8\sigma(1-\sigma)}.
\end{equation}
Furthermore, using \eqref{eq:demandderivduo0}, the derivatives of the demand functions \btt{are determined} as:
\begin{equation}
\label{eq:demandderivcase1}
    \frac{\partial D(\la,\lb)}{\partial \la}=\frac{\partial D(\lb,\la)}{\partial \lb}=\frac{-6\sigma-4\la+2+2\delta}{8\sigma(1-\sigma)}.
\end{equation}
In an equilibrium, both firms should satisfy the first order condition (FOC). We show that the FOC can not hold for both of the firms. Define $\lambda_f=\frac{\lambda_{ij}^f}{\lmax}$ for firm $f$ and let FOC for firm 2 hold:
\begin{equation}
    \frac{\partial D(\lb,\la)}{\partial \lb}(\lb-\lambda_2)+D(\lb,\la)=0.
\end{equation}
Using \eqref{eq:demand1case1}, \eqref{eq:demand2case1}, and \eqref{eq:demandderivcase1}, we can rewrite the above equation as:
\begin{equation}
\label{eq:case1foc}
\begin{split}
        &\frac{\partial D(\la,\lb)}{\partial \la}(\lb-\la+\la-\lambda_1+\lambda_1-\lambda_2)\\
        &+D(\la,\lb)-\frac{4\sigma(1-\sigma-\delta)}{8\sigma(1-\sigma)}+\frac{4\sigma(1-\sigma+\delta)}{8\sigma(1-\sigma)}\\
        &=\frac{\partial D(\la,\lb)}{\partial \la}(\la-\lambda_1)+D(\la,\lb)+\frac{\delta}{1-\sigma}\\
        &+\frac{\partial D(\la,\lb)}{\partial \la}(-\delta-(\lambda_2-\lambda_1))=0.
\end{split}
\end{equation}
To proceed, we use the following lemma:
\begin{lemma}
\label{lem:dualbounds}
    Let $|\delta|\leq 1-\sigma$, $\la-\lb=\delta$, and $\la,\lb\leq 1-\sigma$. If the prices satisfy the FOC, then the following inequality holds:
    \begin{equation}
        |\lambda_1-\lambda_2|\leq(2-\sigma(3\sigma-2))|\delta|
    \end{equation}
\end{lemma}
\begin{proof}
Using \eqref{eq:demand2case1} and \eqref{eq:demandderivcase1}, one can \btt{state} the FOC for a given price $\ell_{-f}=\ell_f+\delta$  to get a quadratic equation in $\ell_f$. This equation has two solutions, one of which is infeasible. Hence, we get the optimal price $\ell_f$ as:
\begin{equation}
\label{eq:case1optprice}
    \ell_f=\frac{3-5\sigma+2\lambda_f-3\delta+\sqrt{\Delta}}{8},
\end{equation}
where
$$
\Delta=(\delta+2\lambda_f+7\sigma-1)^2+32\sigma(\delta-2\sigma+1).
$$
We compute the change in the optimal price with respect to the dual variable as:
\begin{equation}
\label{eq:case1dlambda}
        \frac{\partial \ell_f}{\partial \lambda_f}=\frac{1}{4}\left(1+\frac{\delta+2\lambda_f+7\sigma-1}{\sqrt{\Delta}}\right)
\end{equation}
The goal is to lower bound $\frac{\partial \ell_{f}}{\partial \lambda_f}$. In order to do so, we study how $\frac{\partial \ell_f}{\partial \lambda_f}$ changes with $\lambda_f$ and $\delta$. We first observe that
$$
\frac{\partial^2 \ell_f}{\partial \lambda_f^2}<0,
$$
hence the we need to maximize $\lambda_f$ in order to minimize $\frac{\partial \ell_{f}}{\partial \lambda_f}$. Since $\ell_f$ is constrained to be less than $1-\sigma$, by upper bounding the expression in \eqref{eq:case1optprice} we get a bound on $\lambda_f$ as:
\begin{equation}
    \lambda_f\leq\frac{\delta^2+\delta(4-8\sigma)+3(\sigma-1)^2}{2(\delta+\sigma+1)}
\end{equation}
Next, we plug the upper bound on $\lambda_f$ to \eqref{eq:case1dlambda} to get an expression that is only dependent on $\sigma$ and $\delta$ and compute the partial derivative with respect to $\delta$ to get:
$$
\frac{\partial}{\partial \delta}\frac{\partial \ell_f}{\partial \lambda_f}<0,
$$
hence we maximize $\delta$ in order to minimize $\frac{\partial \ell_f}{\partial \lambda_f}$. We set $\delta=1-\sigma$ to get:
\begin{equation}
    \frac{\partial \ell_f}{\partial \lambda_f}\geq\frac{1}{2-\sigma(3\sigma-2)}.
\end{equation}
Above inequality hods for all $\ell_f\leq1-\sigma$. Since $\la,\lb\leq1-\sigma$, this means:
\begin{equation}
   \Big|\frac{\la-\lb}{\lambda_1-\lambda_2}\Big|\geq \frac{\la-\lb}{\lambda_1-\lambda_2}\geq\frac{1}{2-\sigma(3\sigma-2)}.
\end{equation}
Plugging $\la-\lb=\delta$ concludes the proof.
\end{proof}
Going back to \eqref{eq:case1foc}, we rearrange:
\bt{\begin{align*}
    &\frac{\partial D(\la,\lb)}{\partial \la}(\la-\lambda_1)+D(\la,\lb)\\&=-\frac{\delta}{1-\sigma}-\frac{\partial D(\la,\lb)}{\partial \la}(-\delta-(\lambda_2-\lambda_1))\\
    &\overset{\textnormal{Lemma~\ref{lem:dualbounds}}}{\leq}-\frac{\delta}{1-\sigma}-\frac{\partial D(\la,\lb)}{\partial \la}\delta(-1-(\sigma(3\sigma-2)-2))\\
    &=-\frac{\delta}{1-\sigma}-\delta(1-\sigma(3\sigma-2))\frac{-6\sigma-4\la+2+2\delta}{8\sigma(1-\sigma)}\\
    &=\frac{\delta}{8\sigma(1-\sigma)}(-8\sigma+(1-\sigma(3\sigma-2))
    (6\sigma+4\la-2-2\delta))\\
    &\overset{\la\leq 1-\sigma}{\leq}\frac{\delta}{8\sigma(1-\sigma)}(-8\sigma+(1-\sigma(3\sigma-2))
    (2+2\sigma-2\delta))\\
    &\leq\frac{\delta}{8\sigma(1-\sigma)}(-8\sigma+(1-\sigma(3\sigma-2))
    (2+2\sigma))\\
    &=-\frac{\delta}{4\sigma(1-\sigma)}(3\sigma^3+\sigma^2+\sigma-1)<0,\quad\forall \sigma\in[1/2,1]\tag{\stepcounter{equation}\theequation}.
\end{align*}
We conclude that FOC for firm 1 does not hold, hence they can not be in an equilibrium.}

\vspace{.2cm}
\textbf{Case 2:} Let $\la,\lb\geq(1-\sigma)$. In this region, the demand function and its derivative for firm $f$ can be \btt{stated} as:
\begin{equation}
\label{eq:demandduo20}
    D(\ell_f,\ell_{-f})=\frac{(1-\sigma+(\ell_{-f}-\ell_f))(3+\sigma-(3\ell_f+\ell_{-f}))}{8\sigma(1-\sigma)}
\end{equation}
\begin{equation}
    \label{eq:demandderivduo20}
    \frac{\partial D(\ell_f,\ell_{-f})}{\partial\ell_f}=\frac{-6+2\sigma+6\ell_f-2\ell_{-f}}{8\sigma(1-\sigma)}
\end{equation}
Using the above equations and $\la-\lb=\delta$, we write:
\begin{equation}
\label{eq:case2demand1}
    D(\la,\lb)=\frac{(1-\sigma-\delta)(3+\sigma-4\lb-3\delta)}{8\sigma(1-\sigma)},
\end{equation}
\begin{equation}
\label{eq:case2demand2}
        D(\lb,\la)=\frac{(1-\sigma+\delta)(3+\sigma-4\lb-\delta)}{8\sigma(1-\sigma)},
\end{equation}
\begin{equation}
\label{eq:case2demandderiv1}
    \frac{\partial D(\la,\lb)}{\partial \la}=\frac{-6+2\sigma+4\lb+6\delta}{8\sigma(1-\sigma)},
\end{equation}
\begin{equation}
\label{eq:case2demandderiv2}
    \frac{\partial D(\lb,\la)}{\partial \lb}=\frac{-6+2\sigma+4\lb-2\delta}{8\sigma(1-\sigma)}.
\end{equation}

We follow similar steps as in Case 1 to show that FOC for both firms can not hold. We \btt{state} the FOC for firm 2:
\begin{equation}
\begin{split}
        &\frac{\partial D(\lb,\la)}{\partial \lb}(\lb-\lambda_2)+D(\lb,\la)\\
    &\overset{\eqref{eq:case2demandderiv1},\eqref{eq:case2demandderiv2}}{=}\left(\frac{\partial D(\la,\lb)}{\partial \la}-\frac{8\delta}{8\sigma(1-\sigma)}\right)(\lb-\lambda_2)+D(\lb,\la)\\
    &\overset{\eqref{eq:case2demand1},\eqref{eq:case2demand2}}{=}\left(\frac{\partial D(\la,\lb)}{\partial \la}-\frac{8\delta}{8\sigma(1-\sigma)}\right)(\lb-\lambda_2)\\
    &\quad\quad+D(\la,\lb)+\frac{\delta(8-4\delta-8\lb)}{8\sigma(1-\sigma)}\\
    &=\frac{\partial D(\la,\lb)}{\partial \la}(\lb-\la+\la-\lambda_1+\lambda_1-\lambda_2)\\
    &\quad-\frac{8\delta}{8\sigma(1-\sigma)}(\lb-\lambda_2)+D(\la,\lb)+\frac{\delta(8-4\delta-8\lb)}{8\sigma(1-\sigma)}\\
    &=\frac{\partial D(\la,\lb)}{\partial \la}(\la-\lambda_1)+D(\la,\lb)-\frac{8\delta}{8\sigma(1-\sigma)}(\lb-\lambda_2)\\
    &\quad+\frac{\partial D(\la,\lb)}{\partial \la}(-\delta+\lambda_1-\lambda_2)+\frac{\delta(8-4\delta-8\lb)}{8\sigma(1-\sigma)}\\
    &=0
\end{split}
\end{equation}
For the FOC of firm 1 to hold, the following expression has to be equal to 0:
\begin{equation}
\label{eq:case2expression}
\begin{split}
       &\frac{8\delta}{8\sigma(1-\sigma)}(\lb-\lambda_2)-\frac{\partial D(\la,\lb)}{\partial \la}(-\delta+\lambda_1-\lambda_2)\\
       &-\frac{\delta(8-4\delta-8\lb)}{8\sigma(1-\sigma)}
\end{split}
\end{equation}
We show that the above expression is always less than zero by upper bounding it. To proceed, we use the following lemma:
\begin{lemma}
\label{lem:dualbounds2}
 Let $|\delta|\leq 1-\sigma$, $\la-\lb=\delta$, and $\la,\lb\geq 1-\sigma$. If the prices satisfy the FOC, then the following inequality holds:
 \begin{equation}
 \label{eq:dualbounds2}
     |\lambda_1-\lambda_2|\leq\frac{2|\delta|}{1-\frac{2\overline{\lambda}-2\sigma}{\sqrt{48(1-\sigma)^2+(2\overline{\lambda}-2\sigma)^2}}},
 \end{equation}
 where
 $$\overline{\lambda}=\frac{(3-\sigma)(3\sigma-1)}{4(5-3\sigma)}.$$
\end{lemma}
\begin{proof}
Using \eqref{eq:case2demand2} and \eqref{eq:case2demandderiv2}, one can \btt{state} the FOC for a given price $\ell_{-f}=\ell_f+\delta$ to get a quadratic equation in $\ell_f$. This equation has two solutions, one of which is infeasible. Hence, we get the optimal price $\ell_f$ as:
\begin{equation}
\label{eq:case2optprice}
    \ell_f=\frac{5-3\sigma+3\delta+2\lambda_f-\sqrt{\Delta}}{4}
\end{equation}
where
$$
\Delta=(\delta-\sigma-1+2\lambda_f)^2+12(\sigma-1)^2+12\delta(\delta+2-2\sigma).
$$
Similar to Lemma~\ref{lem:dualbounds}, the goal is to lower bound $\frac{\partial \ell_f}{\partial \lambda_f}$. It is computed as:
\begin{equation}
\label{eq:case2dlambda}
    \frac{\partial \ell_f}{\partial\lambda_f}=\frac{1}{4}\left(2-\frac{2(\delta+\lambda_f-\sigma-1)}{\sqrt{\Delta}}\right)
\end{equation}
In order to minimize the RHS of the above expression, we study how it depends on the variables it is a function of. We first identify that
\begin{equation}
    \frac{\partial^2 \ell_f}{\partial\lambda_f^2}<0,\quad \frac{\partial}{\partial \delta}\frac{\partial \ell_f}{\partial\lambda_f}<0,
\end{equation}
and hence $ \frac{\partial \ell_f}{\partial\lambda_f}$ is minimized when $\delta=1-\sigma$ and $\lambda_f=\frac{(3-\sigma)(3\sigma-1)}{4(5-3\sigma)}$. We plug these expressions to $\frac{\partial \ell_f}{\partial\lambda_f}$ to get:
\begin{equation}
    \frac{\partial \ell_f}{\partial\lambda_f}\geq\frac{1}{2}\left(1-\frac{2\overline{\lambda}-2\sigma}{\sqrt{48(1-\sigma)^2+(2\overline{\lambda}-2\sigma)^2}}\right),
\end{equation}
where $\overline{\lambda}\eqdef\frac{(3-\sigma)(3\sigma-1)}{4(5-3\sigma)}$.

The above inequality holds for all $\ell_f\geq 1-\sigma$ as long as the FOC holds, hence this means:
\begin{equation}
   \Big|\frac{\la-\lb}{\lambda_1-\lambda_2}\Big|\geq \frac{\la-\lb}{\lambda_1-\lambda_2}\geq\frac{1-\frac{2\overline{\lambda}-2\sigma}{\sqrt{48(1-\sigma)^2+(2\overline{\lambda}-2\sigma)^2}}}{2}.
\end{equation}
Plugging $\ell_1-\ell_2=\delta$ concludes the proof.
\end{proof}
\bt{The upper bound \eqref{eq:dualbounds2} in Lemma~\ref{lem:dualbounds2} is concave in $\sigma$ and is decreasing with $\sigma$ in the interval $[0,1]$. For brevity of exposition, we therefore use a linear upper bound in $\sigma$. It can be shown that
\begin{equation}\label{eq:case2linearbound}
        \frac{2}{1-\frac{2\overline{\lambda}-2\sigma}{\sqrt{48(1-\sigma)^2+(2\overline{\lambda}-2\sigma)^2}}}\leq \frac{9-5\sigma}{4}, \quad\forall \sigma\in[0,1].
\end{equation}
Using \eqref{eq:case2linearbound}, we upper bound \eqref{eq:case2expression}:
\begin{equation}
\label{eq:case2exptoupperbound}
\begin{split}
       &\frac{8\delta}{8\sigma(1-\sigma)}(\lb-\lambda_2)-\frac{\partial D(\la,\lb)}{\partial \la}(-\delta+\lambda_1-\lambda_2)\\
       &-\frac{\delta(8-4\delta-8\lb)}{8\sigma(1-\sigma)}\\
       &< \frac{\delta}{8\sigma(1-\sigma)}\Big(8(\lb-\lambda_2)-(8-4\delta-8\lb)\\
       &\hspace{2cm}-\frac{5-5\sigma}{4}(-6+2\sigma+4\lb+6\delta)\Big)
\end{split}
\end{equation}
Given $\lambda_2$ and $\delta$, $\lb$ is uniquely determined as in \eqref{eq:case2optprice}.
We plug $\lb$ to \eqref{eq:case2exptoupperbound}, and conjecture that $\eqref{eq:case2exptoupperbound}<0$. That gives:
\begin{equation}
\begin{split}
   \delta(45\sigma+19)+(1-\sigma)(5\sigma-10\lambda_2+53)<(5\sigma+11)\sqrt{\Delta}
\end{split}
\end{equation}
We take the square of both sides, collect terms on LHS, and re-state the conjecture as:
\begin{equation}
    f(\sigma,\lambda_2,\delta)<0,
\end{equation}
where
\begin{equation}
    \begin{split}
 f(\sigma,\lambda_2,\delta)=&\delta^2(1700\sigma^2+280\sigma-1212)\\
 &+8\delta(25\sigma^2-275\sigma^2+459\sigma-81)\\
 &-300\sigma^4-400\sigma^3+2296\sigma^2-5856\sigma\\
 &-128(5\sigma+3)\lambda_2^2+32\lambda_2\delta(25\sigma^2-30\sigma-27)\\
 &-64\lambda_2(5\sigma^2-46\sigma+9)+1236
    \end{split}
\end{equation}
Our goal is to maximize $f(\sigma,\lambda_2,\delta)$ and show that it is less than $0$. We first identify that $\frac{\partial f(\sigma,\lambda_2,\delta)}{\partial \lambda_2}>0$, hence $f(\sigma,\lambda_2,\delta)$ is maximized when $\lambda_2$ is maximized. We evaluate $f(\sigma,\lambda_2,\delta)$ at $\lambda_2=\frac{(3\sigma-1)(3-\sigma)}{4(5-3\sigma)}$.
We next identify that $\frac{\partial f(\sigma,,\delta)}{\partial \delta}>0$, hence set $\delta=1-\sigma$ to get:
\begin{equation}
\begin{split}
        &f(\sigma,\lambda_2,\delta)\leq f(\sigma,\frac{(3\sigma-1)(3-\sigma)}{4(5-3\sigma)},1-\sigma)\\
        &{=}\frac{8(3\sigma{-}1)(5\sigma{-}7)({-}186+507\sigma{-}209\sigma^2{-}155\sigma^2{+}75\sigma^4)}{(5-3\sigma)^2}.\\
\end{split}
\end{equation}
The above equation has roots at $\sigma\approx-1.8406$, $\sigma=1/3$, $\sigma\approx 0.49744$, $\sigma\approx 1.2599$, and $\sigma=7/5$ and therefore is less than 0 for $\sigma\in[1/2,1]$. Hence, we conclude that $f(\sigma,\lambda_2,\delta)<0$ and the conjecture was true.} Going back, this means that the final expression in \eqref{eq:case2exptoupperbound} is less than zero, which means the expression in \eqref{eq:case2expression} is less than zero, meaning that the FOC of firm 1:
\begin{equation}
    \frac{\partial D(\la,\lb)}{\partial \la}(\la-\lambda_1)+D(\la,\lb)<0.
\end{equation}
Hence, the FOC for firm 1 can not hold, meaning the firms can not be in an equilibrium.

\vspace{.2cm}
\textbf{Case 3:} Let $\lb\leq 1-\sigma$, $\la\geq 1-\sigma$. We show by contradiction that the FOC-satisfying prices can not be $\delta$ apart. We know that if the prices are in equilibrium, FOC holds for both. The optimal prices are given by \eqref{eq:case1optprice} and \eqref{eq:case2optprice} (replacing $\delta$ by $-\delta$) as:
\begin{equation}
    \ell_1=\frac{5-3\sigma-3\delta+2\lambda_1-\sqrt{\Delta_1}}{4},
\end{equation}
\begin{equation}
    \ell_2=\frac{3-5\sigma+2\lambda_2-3\delta+\sqrt{\Delta_2}}{8},
\end{equation}
where
$$
\Delta_1=(-\delta-\sigma-1+2\lambda_1)^2+12(\sigma-1)^2+12\delta(\delta-2+2\sigma),
$$
and
$$
\Delta_2=(\delta+2\lambda_2+7\sigma-1)^2+32\sigma(\delta-2\sigma+1).
$$
But since $\ell_1=\ell_2+\delta$, the following must be true:
\begin{equation}
    \frac{3-5\sigma+2\lambda_2+5\delta+\sqrt{\Delta_2}}{8}=\frac{5-3\sigma-3\delta+2\lambda_1-\sqrt{\Delta_1}}{4}.
\end{equation}
We show that the above equality can not hold by upper bounding the following function, which is the difference between the LHS and the RHS of the above equality:
\begin{equation}
\label{eq:case3fntobound}
    g(\lambda_1,\lambda_2,\delta,\sigma)=\frac{-7+\sigma+11\delta+2\lambda_2-4\lambda_1+\sqrt{\Delta_2}-2\sqrt{\Delta_1}}{8}.
\end{equation}
In order to upper bound the above function, we use the following lemma:
\begin{lemma}\label{lem:dualbounds3}
 Let $|\delta|\leq 1-\sigma$ and $\la-\lb=\delta$. If the prices satisfy the FOC, then the following inequality holds:
 \begin{equation}
     |\lambda_1-\lambda_2|\geq |\delta|
 \end{equation}
\end{lemma}
\begin{proof}
Given $\ell_{-f}=\ell_f+\delta$, the optimal prices for firm $f$ are given by equations \eqref{eq:case1optprice} and \eqref{eq:case2optprice}, for $\ell_f\leq 1-\sigma$ and $\ell_f>1-\sigma$, respectively. Our goal is to upper bound $\frac{\partial \ell_f}{\partial \lambda_f}$, so that we can lower bound the difference between the dual variables, given that the price difference is $\delta$. In proofs of Lemmas~\ref{lem:dualbounds} and \ref{lem:dualbounds2}, we have shown that 
$$
\frac{\partial^2\ell_f}{\partial \lambda_f^2}<0,\quad \frac{\partial}{\partial \delta}\frac{\partial \ell_f}{\partial \lambda_f}<0,
$$
and hence in order to upper bound $\frac{\partial \ell_f}{\partial \lambda_f}$, we set $\lambda_f=0$ and $\delta=0$ in equations \eqref{eq:case1dlambda} and \eqref{eq:case2dlambda}. That gives:
\begin{equation}
    \frac{\partial \ell_f}{\partial \lambda_f}\leq \frac{1}{4}\left(1+\frac{7\sigma-1}{\sqrt{-18\sigma^2+15\sigma+1}}\right),\quad \ell_f\leq 1-\sigma,
\end{equation}
\begin{equation}
    \frac{\partial \ell_f}{\partial \lambda_f}\leq\frac{1}{4}\left(2+\frac{2(\sigma+1)}{\sqrt{13\sigma^2-22\sigma+13}}\right),\quad \ell_f>1-\sigma
\end{equation}
Both of the above equations are increasing with $\sigma$, and are equal to $1$ when $\sigma=1$. Hence:
\begin{equation}
    \frac{\partial \ell_f}{\partial \lambda_f}\leq1,
\end{equation}
or equivalently
\begin{equation}
    \frac{\partial \lambda_f}{\partial \ell_f}\geq 1.
\end{equation}
Since this holds for all $\ell_f$:
\begin{equation}
    \Big|\frac{\lambda_1-\lambda_2}{\ell_1-\ell_2}\Big|\geq\frac{\lambda_1-\lambda_2}{\ell_1-\ell_2}\geq 1.
\end{equation}
Setting $\ell_1-\ell_2=\delta$ concludes the proof.
\end{proof}
Noting that $\frac{\partial g(\lambda_1,\lambda_2,\delta,\sigma)}{\partial \lambda_2}\geq0$ and using Lemma~\ref{lem:dualbounds3} with $\lambda_2\leq\lambda_1-\delta$, we upper bound \eqref{eq:case3fntobound} as:
\begin{equation}
\begin{split}
    g(\lambda_1,\lambda_2,\sigma,\delta)&\leq \hat{g}(\lambda_1,\sigma,\delta)\\
    &=\frac{-7+\sigma+9\delta-2\lambda_1+\sqrt{\hat{\Delta}_1}-2\sqrt{\Delta_2}}{8},
\end{split}
\end{equation}
where
$$
\hat{\Delta}_1=(2\lambda_1-\delta+7\sigma-1)^2+32\sigma(\delta-2\sigma+1).
$$
Our goal is to maximize $\hat{g}$ over its variables, and show that it is always less than $0$. That would mean that when the prices are determined by the FOC, the difference between $\ell_2+\delta$ and $\ell_1$ is always less than zero, which contradicts with $\ell_1-\ell_2=\delta$. We compute the partial derivatives of $\hat{g}$ with respect to $\delta$ and $\lambda_1$ to get:
\begin{equation}
    \frac{\partial \hat{g}(\lambda_1,\sigma,\delta)}{\partial \lambda_1}>0,\quad  \frac{\partial \hat{g}(\lambda_1,\sigma,\delta)}{\partial \delta}>0,
\end{equation}
and hence $\hat{g}(\lambda_1,\sigma,\delta)$ is maximized when $\delta=1-\sigma$ and $\lambda_1=\frac{1}{2}$ (Note that $\lambda_1\leq\frac{(3-\sigma)(3\sigma-1)}{4(5-3\sigma)}\leq\frac{1}{2}$):
\begin{equation}
    \hat{g}(\lambda_1,\sigma,\delta)\leq \hat{g}(1/2,\sigma,1-\sigma){=}\frac{{-}8\sigma{-}1+\sqrt{{-}32\sigma^2+48\sigma+1}}{8}.
\end{equation}
Finally, we observe that $\frac{\partial \hat{g}(1/2,\sigma,1-\sigma)}{\partial \sigma}<0$ \bt{for $\sigma\in[1/2,1]$, and hence $\hat{g}$ is maximized when $\sigma=1/2$. Evaluated at $\sigma=1/2$:
\begin{equation}
    \hat{g}(1/2,1/2,1/2)=\frac{1}{8}(\sqrt{17}-5)<0.
\end{equation}}
Hence, with the FOC satisfying prices, $\ell_2+\delta$ is always less than $\ell_1$, which is a contradiction. This means that FOC can not hold for both firms and thus they can not be in an equilibrium.

We have shown that given $\delta>0$, the FOC can not hold for both of the firms in none of the cases. Hence, the only case when FOC holds for both of the firms is when $\delta=0$, i.e., $\ell_1=\ell_2$. Therefore asymmetric equilibria can not exist. 

\subsection{Proof of Proposition \ref{prop:duomarginalprices}}\label{app:duomarginalprices}
The symmetric duopoly \bt{equilibrium} prices are determined in the proof of Proposition~\ref{prop:unique} (in Appendix~\ref{app:unique}) as:

\begin{equation}
    \label{eq:optduopricesre}
    \lijd=\left\{
\begin{array}{ll}
      \frac{(3-5\sigma)\lmax+2\dijd+\sqrt{\Delta_1^*}}{8} & \frac{\dijd}{\lmax}\leq\frac{3(\sigma-1)^2}{2(\sigma+1)} \\
      \frac{(5-3\sigma)\lmax+2\lambda_{ij}^d-\sqrt{\Delta_2^*}}{4} & o.w. \\

\end{array} 
\right.,
\end{equation}
where
$$\Delta_1^*=4\lmax^2+(2\lambda_{ij}^d+(15\sigma-3)\ell_{\max})(2\lambda_{ij}^d+(1-\sigma)\ell_{\max})$$
and
$$\Delta_2^*=2(\sigma\lmax-\dijd)^2+2(\lmax-\dijd)^2+11(\sigma-1)^2\lmax^2.$$

Both equations in \eqref{eq:optduopricesre} are decreasing functions of $\dijd$. In order to lower bound the difference between the monopoly prices and the duopoly prices, we lower bound the monopoly prices by setting $\lambda_{ij}^{m}=0$ and upper bound the duopoly prices by setting $\lambda_{ij}^{d}=\overline{\lambda}_{ij}$. In order to upper bound the difference, we upper bound the monopoly prices by setting $\lambda_{ij}^{m}=\overline{\lambda}_{ij}$ and lower bound the duopoly prices by setting $\lambda_{ij}^{d}=0$.

\subsection{Proof of Proposition~\ref{cor:pricebounds}}\label{app:pricebounds}
Using $\lambda_{ij}^{m}\leq\overline{\lambda}_{ij}\leq\underset{i,j}{\max} ~\overline{\lambda}_{ij}\leq\dbound$ and $\lambda_{ij}^{m}\geq0$, \eqref{eq:optimalprices} and \eqref{eq:monoboundprices} give:
\begin{equation}
    \frac{(1+\sigma)\lmax}{4}\leq\lijm\leq \frac{7+14\sigma-9\sigma^2}{40-24\sigma}\lmax.
\end{equation}
Furthermore since $\sigma\in[3/5,1]$; $\frac{2}{5}\leq\frac{1+\sigma}{4}$ and $\frac{7+14\sigma-9\sigma^2}{40-24\sigma}\leq\frac{3}{4}$, which completes the part for bounds on monopoly prices. The bounds on duopoly prices is identical using equations in \eqref{eq:optduoprices}.

According to the definitions of $\underline{\ell}^m,\overline{\ell}^m,\underline{\ell}^d,$ and $\overline{\ell}^d$, the bounds on the ratio of prices is:
\begin{equation}
    \frac{\underline{\ell}^d}{\overline{\ell}^m}\leq\frac{\lijd}{\lijm}\leq\frac{\overline{\ell}^d}{\underline{\ell}^m}.
\end{equation}
By plugging in the expressions of $\underline{\ell}^m,\overline{\ell}^m,\underline{\ell}^d,$ and $\overline{\ell}^d$ we get the desired inequality.

\subsection{Proof of Proposition~\ref{cor:demandfunctionbounds}}\label{app:demandbounds}
From \eqref{eq:demandfunctionsmonopoly}, the demand function of the monopoly with the optimal prices is:
\begin{equation}
\label{eq:pfmonodemand}
     D(\lijm,\infty)=\frac{1+\sigma-\frac{2\lijm}{\lmax}}{2\sigma},
\end{equation}
since $\lijm\in[(1-\sigma)\lmax,\sigma\lmax]$ under Assumption~\ref{ass:sigmalambda}.
Plugging in the expressions for $\underline{\ell}^m$ and $\overline{\ell}^m$ and imposing the condition $\sigma\in[3/5,1]$, we get the desired bounds on \eqref{eq:monopolydemandbounds}.

The duopoly demand function for $\lijd=\underline{\ell}^d$ is given by \eqref{eq:demandduoequal} as:
\begin{equation}
    \overline{D}^d=D(\underline{\ell}^d,\underline{\ell}^d)=\frac{1}{2}-\frac{(\frac{2\underline{\ell}^d}{\lmax}+\sigma-1)^2}{8\sigma(1-\sigma)},
\end{equation}
since $\underline{\ell}^d\leq(1-\sigma)\lmax$. The duopoly demand function for $\lijd=\overline{\ell}^d$ is given by \eqref{eq:demandduo2} as:
\begin{equation}
    \underline{D}^d=D(\overline{\ell}^d,\overline{\ell}^d)=\frac{3+\sigma-4\frac{\overline{\ell}^d}{\lmax}}{8\sigma},
\end{equation}
since $\overline{\ell}^d\geq (1-\sigma)\lmax$. By plugging in the expressions for $\underline{\ell}^d$ and $\overline{\ell}^d$ and imposing the condition $\sigma\in[3/5,1]$, we get the desired inequalities in \eqref{eq:duopolydemandbounds}.

According to the definitions of $\underline{D}^m,\overline{D}^m,\underline{D}^d,$ and $\overline{D}^d$, the bounds on the ratio of demand functions is:
\begin{equation}
    \frac{\underline{D}^d}{\overline{D}^m}\leq\frac{D^d}{D^m}\leq\frac{\overline{D}^d}{\underline{D}^m}.
\end{equation}
By plugging in the expressions of $\underline{D}^m,\overline{D}^m,\underline{D}^d,$ and $\overline{D}^d$ and using the condition $\sigma\in[3/5,1]$, we get the desired inequality in \eqref{eq:monoduodemandratio}.
\subsection{Proof of Proposition~\ref{cor:profitbounds}}\label{app:profitbounds}
The total profits generated in monopoly is given by \eqref{eq:monoprofits}. Accordingly, the profits earned by serving the induced demand between OD pair $(i,j)$ is:
\begin{equation}
\label{eq:profitmonosingleod}
    P^m_{ij}=\frac{\theta_{ij}}{4\sigma\ell_{\max}}(\ell_{\max}(1+\sigma)-2\ell_{ij}^m)^2
\end{equation}

Furthermore, lower optimal monopoly prices generate higher profits according to \eqref{eq:monoprofits}. Hence, the upper bound on $P_{ij}^m$ is given by:
\begin{equation}
\begin{split}
    \frac{\theta_{ij}}{4\sigma\ell_{\max}}(\ell_{\max}(1+\sigma)-2\underline{\ell}^m)^2
    =\theta_{ij} \frac{(1+\sigma)^2}{16\sigma}=\theta_{ij}\overline{P}^m,
\end{split}
\end{equation}
To get the lower bound, we evaluate \eqref{eq:profitmonosingleod} at $\lijm=\overline{\ell}^m$. By using the condition $\sigma\in[3/5,1]$, we get the desired inequality in \eqref{eq:monopolyprofitbounds}.

For the profits generated in duopoly, we first show that lower duopoly \bt{equilibrium} prices generate higher profits. Since \eqref{eq:duopolyflowoptimization} bears a similar form to \eqref{eq:flowoptimization}, the dual objective with the optimal prices and dual variables can be \btt{stated} as (similar to \eqref{eq:dualobjective}):
\begin{equation}
\begin{split}
 \label{eq:duodualobjective}
    P^d&=\sum_{i=1}^\bt{n}\sum_{j=1}^\bt{n}\theta_{ij}D(\lijd,\lijd)\left(\lijd-\lambda_{ij}^{d}\right)\\
    &=\sum_{i=1}^\bt{n}\sum_{j=1}^\bt{n} P_{ij}^d,
\end{split}
\end{equation}
where we define 
\begin{equation}
\label{eq:profitsduosingleod}
 P_{ij}^d=\theta_{ij}D(\lijd,\lijd)\left(\lijd-\lambda_{ij}^{d}\right)
\end{equation}
to be profits earned by serving the induced demand between OD pair $(i,j)$. By taking the derivative of $P_{ij}^d$ with respect to $\lijd$:
\begin{equation}
    \frac{dP^d_{ij}}{d\lijd}=\theta_{ij}\left(\frac{dD(\lijd,\lijd)}{d\lijd}\left(\lijd-\lambda_{ij}^{d}\right)+D(\lijd,\lijd)(1-\frac{d\lambda_{ij}^{d}}{d\lijd})\right)
\end{equation}
From \eqref{eq:optduoprices}, $\frac{d\lijd}{d\lambda_{ij}^{d}}\leq 1$. Hence, $\frac{d\lambda_{ij}^{d}}{d\lijd}\geq 1$. Furthermore, $\lijd\geq\lambda_{ij}^{d}$ according to \eqref{eq:optduoprices}. Finally from \eqref{eq:demandduoequal} and \eqref{eq:demandduo2} (evaluated at $\lija=\lijb=\lijd$), $\frac{dD(\lijd,\lijd)}{d\lijd}\leq 0$. All in all that gives: 
$$
\frac{dP^d_{ij}}{d\lijd}\leq 0,
$$
which means lower duopoly \bt{equilibrium} prices generate higher profits. In order to get the upper bound, we evaluate \eqref{eq:profitsduosingleod} at $\lijd=\underline{\ell}^d$ (and $\dijd=0$ in this case)\footnote{When $\lija=\lijb=\lijd=\underline{\ell}^d\leq(1-\sigma)\lmax$, \eqref{eq:duopolyflowoptimization} is not a convex optimization problem (since the demand function is concave in that region). Hence, strong duality might not hold. However, since we are computing an upper bound on the objective function, the objective value of \eqref{eq:duopolyflowoptimization} is always less than or equal to the objective value of the dual problem given by \eqref{eq:duodualobjective}, due to weak duality. Hence, the upper bound holds and is tight when strong duality holds.}. To get the lower bound, we evaluate $\eqref{eq:profitsduosingleod}$ at $\lijd=\overline{\ell}^d$ (and $\dijd=\dbound$ in this case)\footnote{When $\lija=\lijb=\lijd=\overline{\ell}^d\geq(1-\sigma)\lmax$, \eqref{eq:duopolyflowoptimization} is a convex optimization problem since the demand function is linear in that region. Hence, strong duality holds and the value of \eqref{eq:duodualobjective} is equal to the objective value of \eqref{eq:duopolyflowoptimization}.}. To get the desired inequality at \eqref{eq:duopolyprofitbounds}, we impose $\sigma\in[3/5,1]$.

According to the definitions of $\underline{P}^m,\overline{P}^m,\underline{P}^d,$ and $\overline{P}^d$, the bounds on the ratio of profits earned by serving the induced demand for OD pair $(i,j)$ is:
\begin{equation}
    \frac{\underline{P}^d}{\overline{P}^m}\leq\frac{P^d_{ij}}{P^m_{ij}}\leq\frac{\overline{P}^d}{\underline{P}^m}.
\end{equation}
By plugging in the expressions of $\underline{P}^m,\overline{P}^m,\underline{P}^d,$ and $\overline{P}^d$ and using the condition $\sigma\in[3/5,1]$, we get the desired inequality in \eqref{eq:profitsratiomonoduo}.
\subsection{Proof of Proposition~\ref{cor:surplusbounds}}\label{app:csbounds}
The consumer surplus in monopoly is given by \eqref{eq:monocs}. Accordingly, the consumer surplus of customers requesting a ride between OD pair $(i,j)$ is:
\begin{equation}
\label{eq:csmonosingleod}
    \textnormal{CS}^m_{ij}=\theta_{ij}\frac{\lmax(\sigma^2+\sigma+1)-3\lijm(1+\sigma-\frac{\lijm}{\lmax})}{6\sigma},
\end{equation}
Observe that $\frac{\partial \textnormal{CS}^m_{ij}}{\partial\lijm}=\theta_{ij}\frac{6\frac{\lijm}{\lmax}-3\sigma-3}{6\sigma}\leq0$ for $\lijm\in [(1-\sigma)\lmax,\sigma\lmax]$. Hence, lower optimal monopoly prices generate higher consumer surplus. Since $\lijm\geq\underline{\ell}^m=\frac{1+\sigma}{4}\lmax$, the upper bound on  $\textnormal{CS}^m_{ij}$ is given by:
\begin{align}
\nonumber & \theta_{ij}\frac{\lmax(\sigma^2+\sigma+1)-3\lmax\frac{1+\sigma}{4}(1+\sigma-\frac{1+\sigma}{4})}{6\sigma}\\
       &=\theta_{ij} \frac{7\sigma^2-2\sigma+7}{96\sigma}\lmax=\theta_{ij}\overline{\textnormal{CS}}^m.
\end{align}
Similarly, the lower bound on the consumer surplus is given by evaluating \eqref{eq:csmonosingleod} at $\lijm=\overline{\ell}^m=\frac{7+14\sigma-9\sigma^2}{40-24\sigma}\lmax$.  By using the condition $\sigma\in[3/5,1]$, we get the desired inequality in \eqref{eq:monocsbounds}.

Similar to the monopoly, lower duopoly prices generate higher consumer surplus by inducing more customers and generating higher surplus per customer. For $\lijd=\underline{\ell}^d\leq(1-\sigma)\lmax$, the upper bound on the consumer surplus of customers requesting a ride between OD pair $(i,j)$ is computed as:
\begin{equation}
    \begin{split}
        2\theta_{ij}\Big(&\int_{\frac{\lmax}{2}}^\frac{\underline{\ell}^d}{1-\sigma}\int_{\frac{\underline{\ell}^d-(1-\sigma)y}{\sigma}}^{\lmax}(\sigma x+(1-\sigma)y-\underline{\ell}^d)dxdy\\
        &+\int_{\frac{\underline{\ell}^d}{1-\sigma}}^{\lmax}\int_{0}^{\lmax}(\sigma x+(1-\sigma)y-\underline{\ell}^d)dxdy\Big)\\
        &
   =\frac{\theta_{ij}\lmax}{24\sigma(1-\sigma)}\Big((2\sigma)^3-(\sigma+1-2\frac{\underline{\ell}^d}{\lmax})^3\\
    &\hspace{1.5cm}-24\sigma(1-\frac{\underline{\ell}^d}{\lmax})(\sigma-1+\frac{\underline{\ell}^d}{\lmax})\Big)\\
    &=\theta_{ij}\overline{\textnormal{CS}}^d.
    \end{split}
\end{equation}
where the factor 2 is due to the symmetry of two firms.

For $\lijd=\overline{\ell}^d=\frac{1+\sigma}{4}\lmax\geq(1-\sigma)\lmax$, the lower bound is computed as:
\begin{equation}\begin{split}
        &2\theta_{ij}\int_{\frac{\lmax}{2}}^{\lmax}\int_{\frac{\overline{\ell}^d-(1-\sigma)y}{\sigma}}^{\lmax}(\sigma x+(1-\sigma)y-\overline{\ell}^d)dxdy\\
        &=\theta_{ij}\frac{\sigma^2-2\sigma+13}{96\sigma}\lmax=\theta_{ij}\underline{\textnormal{CS}}^d
\end{split}
\end{equation}
By using the condition $\sigma\in[3/5,1]$, we get the desired inequality in \eqref{eq:duocsbound}.

According to the definitions of $\underline{\textnormal{CS}}^m,\overline{\textnormal{CS}}^m,\underline{\textnormal{CS}}^d,$ and $\overline{\textnormal{CS}}^d$, the bounds on the ratio of consumer surplus of customers requesting ride between OD pair $(i,j)$ is:
\begin{equation}
    \frac{\underline{\textnormal{CS}}^d}{\overline{\textnormal{CS}}^m}\leq\frac{\textnormal{CS}^d_{ij}}{\textnormal{CS}^m_{ij}}\leq\frac{\overline{\textnormal{CS}}^d}{\underline{\textnormal{CS}}^m}.
\end{equation}
By plugging in the expressions of $\underline{\textnormal{CS}}^m,\overline{\textnormal{CS}}^m,\underline{\textnormal{CS}}^d,$ and $\overline{\textnormal{CS}}^d$ and using the condition $\sigma\in[3/5,1]$, we get the desired inequality in \eqref{eq:csratiobounds}.

\subsection{MPC with Dynamic Prices in Duopoly}\label{app:mpc}
One possible way to model the real-time duopoly pricing is an alternating-move duopoly game. Specifically, every even $t_0$, firm 1 sets new prices and executes fleet decisions, whereas firm 2 only executes fleet decisions while keeping prices same as the previous time period. Every odd $t_0$, firm 2 sets new prices and executes fleet decisions, whereas firm 1 only executes fleet decisions while keeping prices same as the previous time period. Furthermore, every even $t_0$, firm 1 is able to observe firm 2's prices at the planning time, however, the future prices of firm 2 depend on firm 2's future states, which is unavailable to firm 1. Every odd $t_0$ however, since firm 2 will set the prices, firm 1 is oblivious to what firm 2's prices will be, including the planning time. One possible way of planning for these uncertainties would be to assume that firm 2's unknown prices would be the symmetric duopoly \bt{equilibrium} prices and determine the best strategy accordingly. In respect to these modeling specifications, we can formulate MPC optimization problem with dynamic prices in the duopoly with slight modifications to \eqref{eq:mpcdymono}. In particular, every even $t_0$ firm 1 solves \eqref{eq:mpcdymono} with
\begin{align}
    &D(\ell_{ijt_0}^1,\infty)\leftarrow D(\ell_{ijt_0}^1,\ell_{ijt_0}^2),\\
    &D(\ell_{ijt}^1,\infty)\leftarrow D(\ell_{ijt}^1,\ell_{ij}^d),\quad \forall t>t_0,\\
    &\ell_{ijt}^1=\ell_{ijt-1}^1, \quad\forall i,j\in\bt{\cal N},\;\forall t=t_0+2k-1, k\in \mathbb{Z}^+,
\end{align}
where $\ell_{ijt_0}^2=\ell_{ijt_0-1}^2$. Every odd $t_0$, firm 1 solves \eqref{eq:mpcdymono} with
\begin{align}
    &D(\ell_{ijt}^1,\infty)\leftarrow D(\ell_{ijt}^1,\ell_{ij}^d),\quad \forall t\geq t_0\\
    &\ell_{ijt}^1=\ell_{ijt-1}^1, \quad\forall i,j\in\bt{\cal N},\;\forall t=t_0+2k-2, k\in \mathbb{Z}^+.
\end{align}
The same method is applied for firm 2 with odd/even $t$ switched.

\end{document}